\UseRawInputEncoding
\documentclass[12pt]{article}
\usepackage{amsmath, amssymb, amsthm, graphicx, hyperref, array, booktabs, microtype}
\usepackage[margin=0.85in]{geometry}
\newtheorem{theorem}{Theorem}[section]
\newtheorem{definition}[theorem]{Definition}
\newtheorem{remark}[theorem]{Remark}
\newtheorem{lemma}[theorem]{Lemma}
\newtheorem{corollary}[theorem]{Corollary}
\newtheorem{proposition}[theorem]{Proposition}
\newtheorem{example}[theorem]{Example}

\hypersetup{
    colorlinks=true,
    linkcolor=black,
    citecolor=black,
    urlcolor=blue,
}

\title{Second Order Zarankiewicz Number}
\author{Johan L\"{o}fberg\footnote{Division of Automatic Control, Link\"{o}ping University, Sweden. ({\tt johan.lofberg@liu.se})}
\and
Liqun Qi\footnote{Jiangsu Provincial Scientific Research Center of Applied Mathematics, Nanjing 211189, China.
Department of Applied Mathematics, The Hong Kong Polytechnic University, Hung Hom, Kowloon, Hong Kong.
({\tt maqilq@polyu.edu.hk})}
}
\date{\today}

\begin{document}

\maketitle

\begin{abstract}
We introduce the \emph{second order Zarankiewicz number} $z_2(m,n)$ for irreducible doubly simple biquadratic forms with $|E_1|=z(m,n)$, together with the intermediate recursive-line and signed parameters $z_{RL}(m,n)$ and $z_{SL}(m,n)$. They satisfy the unconditional hierarchy
\[
\operatorname{BSR}(m,n) \ge z_2(m,n) \ge z_{SL}(m,n) \ge z_{RL}(m,n) \ge z_{wL}(m,n) \ge z(m,n),
\]
where $z_{RL}$ is defined by the strengthened recursive rectangle criterion $(RW3^+)$ together with the conditions $(S)$ and $C_4$-freeness of $G_1$, and $z_{SL}$ is defined by the signed criterion $(RW3^\pm)$. We show that $(RW3^+)$ is sound and strictly weaker than the literal weak cross-cell test $(W3)$ on the weak-admissible class, and that $(RW3^\pm)$ is likewise sound. At the smallest four-column cases we obtain
\[
z_2(5,4)=z_{SL}(5,4)=z_{RL}(5,4)=13>12=z_{wL}(5,4),
\]
and
\[
z_2(6,4)=z_{SL}(6,4)=z_{RL}(6,4)=16>14=z_{wL}(6,4).
\]
In three columns this yields
\[
z_2(m,3)=z_{RL}(m,3)=2m \qquad \text{for all } m\ge 3,
\]
with strict separation from $z_{wL}(m,3)$ for every $m\ge10$, and exact gap
$\lfloor(m-3)/3\rfloor$ for $m\ge16$. Further finite computations give
$z_{RL}(5,5)=17$, $z_{RL}(7,4)=19$, and
$z_{RL}(7,7)\ge32>28=z_{wL}(7,7)$. Along $N=2p$ with $p$ an odd prime,
we obtain the cubic asymptotic separation
\[
z_2\!\left(\binom{N}{2},N\right)-z_{wL}\!\left(\binom{N}{2},N\right)\ge \left(\frac{1}{16}-o(1)\right)N^3.
\]
For the exceptional incidence case $p=3$ we prove $z_{SL}(15,6)=z_2(15,6)=60$. Concurrent work of Chen and Chen, using the recursive-line definition introduced here, incorporates the finite recursive-line values of this manuscript and establishes further exact four-column values together with an eventual formula for all $m\ge15$. The conjectural equality $z_2=z_{RL}$ is supported by the exact two-column, three-column, four-column, and odd-prime incidence families.


\end{abstract}

\section{Introduction}

The study of biquadratic forms and their sum-of-squares (SOS) representations has deep connections to both algebraic geometry \cite{clz95} and extremal combinatorics \cite{bo04}. For an $m \times n$ biquadratic form
\[
P(\mathbf{x},\mathbf{y}) = \sum_{i,k=1}^m \sum_{j,l=1}^n a_{ijkl} x_i x_k y_j y_l,
\]
with symmetry conditions $a_{ijkl} = a_{kjil} = a_{klij}$, the SOS rank $\operatorname{SOS}(P)$ is the minimum number of squares needed in an SOS decomposition. The biquadratic SOS rank $\operatorname{BSR}(m,n)$ is the maximum SOS rank among all $m \times n$ SOS biquadratic forms \cite{qi1, qi3}.

A remarkable connection between SOS rank and extremal graph theory was established in \cite{qi1}: for every $C_4$-free bipartite graph $G_1 = (S,T,E_1)$, the associated doubly simple biquadratic form
\[
P_G(\mathbf{x},\mathbf{y}) = \sum_{(i,j)\in E_1} x_i^2 y_j^2 + \sum_{(i,j;k,l)\in E_2} (x_i y_j + x_k y_l)^2
\]
has SOS rank at most $|E_1| + |E_2|$. When the graph satisfies certain admissibility conditions, this form is irreducible, meaning $\operatorname{SOS}(P_G) = |E_1| + |E_2|$. This led to the definition of the limited augmented Zarankiewicz number $z_L(m,n)$ \cite{qi1} and, more recently, the weak limited augmented Zarankiewicz number $z_{wL}(m,n)$ \cite{qi3}.

The weak framework relaxes the original admissibility conditions by replacing the original Condition 2 with two weaker requirements: acyclicity of the dependency graph of nondegenerate 2-edges (with complementary 2-cycles allowed), and a local prohibition that the two opposite cells of a nondegenerate 2-edge cannot both be 1-edges. Condition 3 is replaced by the requirement that every vertex-disjoint pair with at least one 2-edge has an unoccupied opposite cell \cite{qi3}.

However, these weak conditions, while sufficient for irreducibility, are not necessary. There exist irreducible doubly simple forms that violate the weak conditions, particularly $(W3)$. These forms achieve irreducibility through a more subtle mechanism: recursive propagation of orthogonality through coefficient identities, rather than through literal unoccupied cells.

In this paper, we introduce the \emph{second order Zarankiewicz number} $z_2(m,n)$ to capture exactly this algebraic phenomenon. We define it as the maximum SOS rank of an $m \times n$ irreducible doubly simple biquadratic form with $|E_1|=z(m,n)$. We also introduce the intermediate recursive-line parameter $z_{RL}(m,n)$ obtained by imposing $(S)$, $C_4$-freeness of $G_1$, and the strengthened recursive rectangle test $(RW3^+)$. The unconditional hierarchy is
\[
\operatorname{BSR}(m,n) \ge z_2(m,n) \ge z_{SL}(m,n) \ge z_{RL}(m,n) \ge z_{wL}(m,n) \ge z(m,n),
\]
where the signed parameter $z_{SL}$ is defined later by the criterion $(RW3^\pm)$. The results of this paper show that both $z_{RL}$ and $z_2$ can be strictly larger than $z_{wL}$. In two columns, however, everything collapses to the exact identity
\[
\operatorname{BSR}(m,2)=z_2(m,2)=z_{RL}(m,2)=z(m,2)=m+1
\]
for every $m$.

Our main contributions are these.
\begin{itemize}
    \item We formalize recursive orthogonality propagation by the intrinsic closure $(RW3)$ and its strengthened version $(RW3^+)$, and prove that $(RW3^+)$ is sound and strictly weaker than the literal weak cross-cell test $(W3)$ on the weak-admissible class.
    \item We prove the first exact four-column equalities beyond the weak framework:
    \[
    z_2(5,4)=z_{SL}(5,4)=z_{RL}(5,4)=13>12=z_{wL}(5,4),
    \]
    \[
    z_2(6,4)=z_{SL}(6,4)=z_{RL}(6,4)=16>14=z_{wL}(6,4).
    \]
    Concurrent work of Chen and Chen~\cite{cc26mx4}, using the recursive-line definition introduced here, incorporates the finite recursive-line values obtained in this manuscript and establishes further exact values and an eventual four-column formula.
    Hence the weak admissibility conditions are sufficient but not necessary for irreducibility, while the recursive-line parameter remains exact in both cases.
    \item We show that the recursive-line parameter can also improve on the weak framework in the three-column regime and at the standard $7\times 7$ benchmark. In particular,
    \[
    z_2(m,3)=z_{RL}(m,3)=2m \qquad \text{for all } m\ge 3,
    \]
    with strict inequality over the exact weak values for every $m\ge 10$; for $m\ge 16$ the exact gap is $\lfloor (m-3)/3\rfloor$ by \cite{qlc26}, and
    \[
    z_{RL}(7,7)\ge 32>28=z_{wL}(7,7).
    \]
    \item Exhaustive searches with independent exact verifiers establish
    $z_{RL}(5,5)=17$ and $z_{RL}(7,4)=19$. Explicit witnesses and complete
    search records are available in the accompanying repository.
    \item Along the odd-prime subsequence of the complete-graph incidence family we prove the cubic asymptotic separation
    \[
    z_2\left(\binom{N}{2}, N\right) - z_{wL}\left(\binom{N}{2}, N\right)
    \ge \left(\frac{1}{16} - o(1)\right)N^3,
    \]
    so the gap between the second-order and weak frameworks persists in infinite families.
    \item We introduce the signed recursive criterion $(RW3^\pm)$ and the signed parameter $z_{SL}$. In particular, for the exceptional incidence case $p=3$ we prove
    \[
    z_{SL}(15,6)=z_2(15,6)=60.
    \]
\end{itemize}

\section{Preliminaries}

\subsection{Biquadratic Forms and SOS Rank}

Let $P(\mathbf{x},\mathbf{y})$ be an $m \times n$ biquadratic form. We say $P$ is \emph{positive semidefinite} (PSD) if $P(\mathbf{x},\mathbf{y}) \ge 0$ for all $\mathbf{x} \in \mathbb{R}^m, \mathbf{y} \in \mathbb{R}^n$. It is a \emph{sum of squares} (SOS) if there exist bilinear forms $f_1,\dots,f_r$ such that
\[
P(\mathbf{x},\mathbf{y}) = \sum_{t=1}^r f_t(\mathbf{x},\mathbf{y})^2.
\]
The minimum such $r$ is the SOS rank $\operatorname{SOS}(P)$. The biquadratic SOS rank $\operatorname{BSR}(m,n)$ is the maximum SOS rank among all $m \times n$ SOS biquadratic forms \cite{qi1, qi3}.

\subsection{Augmented Bipartite Graphs}

Let $G_1 = (S,T,E_1)$ be an $m \times n$ bipartite graph with $S = [m]$, $T = [n]$. An \emph{augmented bipartite graph} $G = (S,T,E_1 \cup E_2)$ adds a set $E_2$ of 2-edges. A 2-edge $(i,j;k,l)$ has halves $(i,j)$ and $(k,l)$. It is:
\begin{itemize}
    \item \emph{nondegenerate} if $i \ne k$ and $j \ne l$,
    \item \emph{row-degenerate} if $i = k$ and $j \ne l$,
    \item \emph{column-degenerate} if $i \ne k$ and $j = l$.
\end{itemize}
Two nondegenerate 2-edges are \emph{complementary} if they are the two diagonals of one genuine rectangle. An augmented bipartite graph is \emph{limited} if $|E_1|=z(m,n)$.
The \emph{simplicity condition} $(S)$ requires that no cell is used by more than one edge.

The associated \emph{doubly simple biquadratic form} is
\[
P_G(\mathbf{x},\mathbf{y}) = \sum_{(i,j)\in E_1} x_i^2 y_j^2 + \sum_{(i,j;k,l)\in E_2} (x_i y_j + x_k y_l)^2. \tag{1}
\]
We say $P_G$ is \emph{irreducible} if $\operatorname{SOS}(P_G) = |E_1| + |E_2|$.

\subsection{Weak Admissibility}

We recall the weak admissibility conditions from \cite{qi3}, with the simplicity condition $(S)$ made explicit.

{
For a nondegenerate 2-edge $e=(i,j;k,l)$, its two \emph{opposite cells} are $(i,l)$ and $(k,j)$. For any selected edge $e$, let $R(e)$ and $C(e)$ denote the sets of row and column indices appearing in its support. For two vertex-disjoint selected edges $e,f$, define the \emph{opposite-cell set} by
\[
O(e,f)=\bigl(R(e)\times C(f)\bigr)\cup \bigl(R(f)\times C(e)\bigr),
\]
with repetitions suppressed. The weak cross-cell test $(W3)$ asks that, whenever two selected edges are vertex-disjoint and at least one is a 2-edge, the set $O(e,f)$ contain an unoccupied cell.

The \emph{dependency graph} of the nondegenerate selected 2-edges has a directed edge $e\to e'$ if both opposite cells of $e$ are occupied and one of them is a half of $e'$. In checking $(W2)$, each mutually complementary pair is first contracted to one vertex, so the internal complementary 2-cycle disappears rather than becoming a self-loop.
}

\begin{definition}[Weak generalized $C_4$-cycle]
An augmented bipartite graph $G = (S,T,E_1 \cup E_2)$ contains a weak generalized $C_4$-cycle if any of the following holds:
\begin{itemize}
    \item \textbf{(S)} The graph is not simple, i.e., some cell is used by more than one edge;
    \item \textbf{(W1)} The 1-edge graph $G_1$ contains a classical $C_4$;
    \item \textbf{(W2)} After every mutually complementary pair of nondegenerate 2-edges is contracted to one vertex, the resulting dependency graph contains a directed cycle;
    \item \textbf{(W2')} There exists a nondegenerate 2-edge $(i,j;k,l)$ such that both opposite cells $(i,l)$ and $(k,j)$ are 1-edges;
    \item \textbf{(W3)} There exists a vertex-disjoint pair of edges where at least one edge is a 2-edge, for which every cell in the corresponding opposite-cell set $O(e,f)$ is occupied.
\end{itemize}
If none of these occurs, $G$ is called \emph{weak admissible}.
If $|E_1|=z(m,n)$, we call $G$ a \emph{limited augmented bipartite graph}.
\end{definition}

The weak limited augmented Zarankiewicz number $z_{wL}(m,n)$ is the maximum total number of edges $|E_1| + |E_2|$ over all weak admissible limited augmented bipartite graphs with $|E_1| = z(m,n)$, where $z(m, n)$ is the classical Zarankiewicz number \cite{chm24, gu69, kst54, ni10, re58, za51}.

\section{The Second Order Zarankiewicz Number}

\subsection{Definition}

\begin{definition}[Second Order Zarankiewicz number]
The \emph{second order Zarankiewicz number} $z_2(m,n)$ is the maximum SOS rank of an $m \times n$ irreducible doubly simple biquadratic form (1) with $|E_1| = z(m,n)$.

Equivalently,
\[
\begin{aligned}
z_2(m,n) = \max \bigl\{ |E_1| + |E_2| :\ &|E_1|=z(m,n),\ G \text{ satisfies } (S),\\
&G_1 \text{ is } C_4\text{-free},\ P_G \text{ is irreducible} \bigr\},
\end{aligned}
\]
where $G = (S,T,E_1 \cup E_2)$ is a limited augmented bipartite graph.
\end{definition}

\begin{proposition}[Universal cell bound]\label{prop:universal-cell-bound}
{
For every irreducible limited doubly simple form $P_G$ on an $m\times n$ grid,
\[
\operatorname{SOS}(P_G)=|E_1|+|E_2|\le \left\lfloor \frac{mn+z(m,n)}{2}\right\rfloor.
\]
}
\end{proposition}

\begin{proof}
If four 1-edges occupy one rectangle, write the corresponding monomials so that $(a,d)$ and $(b,c)$ are the two diagonals. Since $ad=bc$,
\[
a^2+b^2+c^2+d^2=(a+d)^2+(b-c)^2,
\]
so the displayed decomposition is reducible. Hence irreducibility forces $E_1$ to be $C_4$-free, and therefore $|E_1|\le z(m,n)$. Simplicity gives $|E_1|+2|E_2|\le mn$. Combining these inequalities,
\[
2(|E_1|+|E_2|)=|E_1|+(|E_1|+2|E_2|)\le z(m,n)+mn.
\]
Taking floors gives the stated bound.
\end{proof}

\begin{corollary}\label{cor:z2-upper-bound}
For all $m,n\ge 2$,
\[
z_2(m,n)\le \left\lfloor \frac{mn+z(m,n)}{2}\right\rfloor.
\]
If $m\ge \binom{n}{2}$, then $z(m,n)=m+\binom{n}{2}$ by \cite{culik56}, so
\[
z_2(m,n)\le \left\lfloor \frac{(n+1)m+\binom{n}{2}}{2}\right\rfloor.
\]
In particular, for $m\ge 3$,
\[
z_2(m,3)\le 2m+1.
\]
\end{corollary}

\begin{proof}
This is immediate from Proposition~\ref{prop:universal-cell-bound}, since every form counted by $z_2(m,n)$ is an irreducible limited doubly simple form with $|E_1|=z(m,n)$.
\end{proof}

\subsection{The Hierarchy}

We first record the basic inequalities between the classical, weak, and second-order parameters.

\begin{theorem}
For all $m,n \ge 2$,
\[
\operatorname{BSR}(m,n) \ge z_2(m,n) \ge z_{wL}(m,n) \ge z(m,n).
\]
\end{theorem}

\begin{proof}
\begin{enumerate}
    \item $\operatorname{BSR}(m,n) \ge z_2(m,n)$, since irreducible doubly simple forms are a subclass of SOS biquadratic forms.
    \item $z_2(m,n) \ge z_{wL}(m,n)$, by the main theorem of \cite{qi3}.
    \item $z_{wL}(m,n) \ge z(m,n)$, since every extremal $C_4$-free graph with $E_2=\emptyset$ is weak admissible.
\end{enumerate}
\end{proof}

\begin{theorem}[Exact three-column upper bound]\label{thm:m3-upper}
For every $m\ge 3$,
\[
z_2(m,3)\le 2m.
\]
\end{theorem}

\begin{proof}
Let $P_G$ be any form counted by $z_2(m,3)$. Then $P_G$ is irreducible, limited, and doubly simple, with
\[
|E_1|=z(m,3)=m+3
\]
by \v{C}ul\'{\i}k's formula. Thus the number of cells outside $E_1$ is
\[
3m-(m+3)=2m-3,
\]
which is odd. Every selected $2$-edge occupies exactly two such cells, and simplicity prevents overlaps, so at least one cell remains unoccupied. If some column contains $h\ge 1$ holes, then at most $m-h$ displayed squares can touch that column, while the remaining displayed squares form an $m\times 2$ SOS biquadratic form after deleting it. Hence
\[
\operatorname{SOS}(P_G)\le m-h+\operatorname{BSR}(m,2)=2m+1-h\le 2m,
\]
using the exact two-column formula $\operatorname{BSR}(m,2)=m+1$ from \cite{qi1}.
Since $P_G$ is irreducible,
\[
|E_1|+|E_2|=\operatorname{SOS}(P_G)\le 2m.
\]
Taking the maximum over all irreducible limited doubly simple $m\times 3$ forms proves the claim.
\end{proof}

\subsection{The weak value at (5,4)}

\begin{proposition}\label{prop:zwl54}
\[
z_{wL}(5,4)=12.
\]
\end{proposition}

\begin{proof}
We first show the lower bound $z_{wL}(5,4)\ge 12$ by an explicit weak admissible construction. Let
\[
E_1=\{(1,3),(1,4),(2,2),(2,4),(3,2),(3,3),(4,1),(4,4),(5,1),(5,3)\}
\]
and
\[
E_2=\{(3,1;3,4),(4,2;5,2)\}.
\]
Then $|E_1|=10=z(5,4)$ and $|E_2|=2$, so the total number of edges is $12$.

The $1$-edge graph is $C_4$-free: its column supports are
\[
C_1=\{4,5\},\qquad C_2=\{2,3\},\qquad C_3=\{1,3,5\},\qquad C_4=\{1,2,4\},
\]
and every pair of these sets intersects in at most one row. Thus $(W1)$ does not occur. Both $2$-edges are degenerate, so $(W2)$ and $(W2')$ are vacuous.

It remains to check $(W3)$. For the row-degenerate edge $(3,1;3,4)$, the only vertex-disjoint edges are
\[
(1,3),\ (2,2),\ (5,3),\ (4,2;5,2),
\]
and their opposite-cell sets contain the empty cells $(1,1)$, $(2,1)$, $(5,4)$, and $(5,4)$, respectively. For the column-degenerate edge $(4,2;5,2)$, the additional vertex-disjoint $1$-edges are
\[
(1,3),\ (1,4),\ (2,4),\ (3,3),
\]
and their opposite-cell sets contain the empty cells $(1,2)$, $(1,2)$, $(5,4)$, and $(4,3)$, respectively. Hence every vertex-disjoint pair with at least one $2$-edge has an unoccupied opposite cell, so $(W3)$ also holds. Therefore this graph is weak admissible, and $z_{wL}(5,4)\ge 12$.

For the upper bound, suppose $G=(S,T,E_1\cup E_2)$ is weak admissible with $|E_1|=z(5,4)=10$ and $|E_1|+|E_2|\ge 13$. Then $|E_2|\ge 3$. Since weak admissibility is preserved when $2$-edges are deleted, it is enough to exclude the case $|E_2|=3$.

Up to row and column permutations, there are exactly three extremal $C_4$-free $5\times 4$ graphs with $10$ ordinary edges. For each representative $E_1$, there are
\[
\binom{10}{6}\cdot 15=3150
\]
ways to choose three pairwise disjoint $2$-edges on the ten cells outside
$E_1$: first choose the six used cells, then pair them into three unordered
pairs. An exact exhaustive computation rejects all $9450$ candidates.
The search and its complete rejection records are available in
\cite{reproducibility}. Therefore no weak admissible configuration with
$|E_2|\ge3$ exists, so $z_{wL}(5,4)\le12$.

Combining the lower and upper bounds gives $z_{wL}(5,4)=12$.
\end{proof}

\subsection{\texorpdfstring{The Gap: $z_2(5,4) > z_{wL}(5,4)$}{The Gap: z2(5,4) > zwL(5,4)}}

The following theorem establishes a strict separation in the hierarchy.

\begin{theorem}
\[
z_2(5,4) > z_{wL}(5,4).
\]
Specifically,
\[
z_2(5,4) \ge 13 > 12 = z_{wL}(5,4).
\]
\end{theorem}

\begin{proof}
We exhibit an explicit $5 \times 4$ irreducible doubly simple form with SOS rank 13.

Let
\[
E_1 = \{(1,2),(1,3),(1,4),(2,1),(2,2),(3,4),(4,1),(4,4),(5,1),(5,3)\}.
\]
This is an extremal $C_4$-free graph with $|E_1| = 10 = z(5,4)$.

Let
\[
E_2 = \{(3,1;4,2), (5,2;3,3), (4,3;2,4)\}.
\]
All three 2-edges are nondegenerate and use six distinct cells. The total number of edges is 13.

The associated doubly simple biquadratic form is
\[
\begin{aligned}
P_G(\mathbf{x},\mathbf{y}) = & x_1^2(y_2^2 + y_3^2 + y_4^2) + x_2^2(y_1^2 + y_2^2) + x_3^2 y_4^2 \\
& + x_4^2(y_1^2 + y_4^2) + x_5^2(y_1^2 + y_3^2) \\
& + (x_3 y_1 + x_4 y_2)^2 + (x_5 y_2 + x_3 y_3)^2 + (x_4 y_3 + x_2 y_4)^2.
\end{aligned}
\]
Immediately, $\operatorname{SOS}(P_G) \le 13$ from the defining representation.

To show irreducibility, we use a vector argument. Suppose $P_G$ has an SOS representation with $r$ squares. Write
\[
P_G = \sum_{t=1}^r f_t^2, \qquad f_t = \sum_{i=1}^5 \sum_{j=1}^4 v_{ij}^{(t)} x_i y_j,
\]
and set $\mathbf{v}_{ij} = (v_{ij}^{(1)}, \dots, v_{ij}^{(r)})^\top \in \mathbb{R}^r$. Coefficient comparison gives the following standard facts:

\begin{enumerate}
    \item For each occupied cell $(i,j)$, $\|\mathbf{v}_{ij}\|^2 = 1$. For each unoccupied cell, $\mathbf{v}_{ij} = 0$.

    \item In this particular construction, two occupied cells in the same row or the same column are orthogonal. Indeed, all selected $2$-edges here are nondegenerate, so no such same-line pair can itself be a selected degenerate $2$-edge.

    \item For every genuine rectangle $(i,k)\times(j,\ell)$ with $i \neq k$ and $j \neq \ell$,
    \[
    \mathbf{v}_{ij} \cdot \mathbf{v}_{k\ell} + \mathbf{v}_{i\ell} \cdot \mathbf{v}_{kj}
    \]
    equals the number of selected 2-edges on the two diagonals of that rectangle.
\end{enumerate}

The occupied cells are:
\[
\begin{array}{c|cccc}
 & 1 & 2 & 3 & 4 \\ \hline
1 & 0 & a & b & c \\
2 & d & e & 0 & w \\
3 & u & 0 & v & f \\
4 & g & u & w & h \\
5 & i & v & j & 0
\end{array}
\]
Here $a,\dots,j$ denote the ten 1-edge vectors and $u,v,w$ denote the three 2-edge vectors:
\[
u = (3,1;4,2), \qquad v = (5,2;3,3), \qquad w = (4,3;2,4).
\]
Thus the four unoccupied cells are $(1,1)$, $(2,3)$, $(3,2)$, and $(5,4)$.

First the three 2-edges resolve immediately. For rows $3,4$ and columns $1,2$, the cell $(3,2)$ is unoccupied, so the rectangle identity gives
\[
\mathbf{v}_{31} \cdot \mathbf{v}_{42} = 1.
\]
Since both vectors are unit vectors, this forces
\[
\mathbf{v}_{31} = \mathbf{v}_{42} =: \mathbf{u}.
\]
Similarly,
\[
\mathbf{v}_{52} = \mathbf{v}_{33} =: \mathbf{v}, \qquad \mathbf{v}_{43} = \mathbf{v}_{24} =: \mathbf{w}.
\]
Hence every SOS representation contains the same thirteen unit edge-vectors
\[
a,b,c,d,e,f,g,h,i,j,u,v,w.
\]
It remains to show that they are pairwise orthogonal.

\noindent
Shared rows and columns give $38$ of the $\binom{13}{2} = 78$ orthogonalities immediately:
\[
\begin{aligned}
a &\perp \{b,c,e,u,v\}, & b &\perp \{c,j,v,w\}, & c &\perp \{f,h,w\},\\
d &\perp \{e,g,i,u,w\}, & e &\perp \{u,v,w\}, & f &\perp \{h,u,v,w\},\\
g &\perp \{h,i,u,w\}, & h &\perp \{u,w\}, & i &\perp \{j,u,v\},\\
j &\perp \{v,w\}, & u &\perp \{v,w\}, & v &\perp w.
\end{aligned}
\]
The remaining $40$ orthogonalities follow from zero-right-hand-side rectangle identities. In each line below,
the target inner product plus the listed companion term is zero; the companion is either already zero by the previous list,
contains an unoccupied cell, or has been proved zero on an earlier line.

\begin{center}
\small
\begin{tabular}{lll}
Target & Rectangle & Companion \\
\hline
$\langle a,d\rangle$ & $(1,2)\times(1,2)$ & $0$ \\
$\langle b,d\rangle$ & $(1,2)\times(1,3)$ & $0$ \\
$\langle c,d\rangle$ & $(1,2)\times(1,4)$ & $0$ \\
$\langle b,e\rangle$ & $(1,2)\times(2,3)$ & $0$ \\
$\langle b,u\rangle$ & $(1,3)\times(1,3)$ & $0$ \\
$\langle c,u\rangle$ & $(1,3)\times(1,4)$ & $0$ \\
$\langle a,f\rangle$ & $(1,3)\times(2,4)$ & $0$ \\
$\langle a,g\rangle$ & $(1,4)\times(1,2)$ & $0$ \\
$\langle b,g\rangle$ & $(1,4)\times(1,3)$ & $0$ \\
$\langle c,g\rangle$ & $(1,4)\times(1,4)$ & $0$ \\
$\langle a,w\rangle$ & $(1,4)\times(2,3)$ & $\langle b,u\rangle$ \\
$\langle a,h\rangle$ & $(1,4)\times(2,4)$ & $\langle c,u\rangle$ \\
$\langle b,h\rangle$ & $(1,4)\times(3,4)$ & $\langle c,w\rangle$ \\
$\langle a,i\rangle$ & $(1,5)\times(1,2)$ & $0$ \\
$\langle b,i\rangle$ & $(1,5)\times(1,3)$ & $0$ \\
$\langle c,i\rangle$ & $(1,5)\times(1,4)$ & $0$ \\
$\langle a,j\rangle$ & $(1,5)\times(2,3)$ & $\langle b,v\rangle$ \\
$\langle c,v\rangle$ & $(1,5)\times(2,4)$ & $0$ \\
$\langle c,j\rangle$ & $(1,5)\times(3,4)$ & $0$ \\
$\langle d,v\rangle$ & $(2,3)\times(1,3)$ & $0$ \\
$\langle d,f\rangle$ & $(2,3)\times(1,4)$ & $\langle u,w\rangle$ \\
$\langle e,f\rangle$ & $(2,3)\times(2,4)$ & $0$ \\
$\langle e,g\rangle$ & $(2,4)\times(1,2)$ & $\langle d,u\rangle$ \\
$\langle d,h\rangle$ & $(2,4)\times(1,4)$ & $\langle g,w\rangle$ \\
$\langle e,h\rangle$ & $(2,4)\times(2,4)$ & $\langle u,w\rangle$ \\
$\langle e,i\rangle$ & $(2,5)\times(1,2)$ & $\langle d,v\rangle$ \\
$\langle d,j\rangle$ & $(2,5)\times(1,3)$ & $0$ \\
$\langle i,w\rangle$ & $(2,5)\times(1,4)$ & $0$ \\
$\langle e,j\rangle$ & $(2,5)\times(2,3)$ & $0$ \\
$\langle g,v\rangle$ & $(3,4)\times(1,3)$ & $\langle u,w\rangle$ \\
$\langle f,g\rangle$ & $(3,4)\times(1,4)$ & $\langle h,u\rangle$ \\
$\langle h,v\rangle$ & $(3,4)\times(3,4)$ & $\langle f,w\rangle$ \\
$\langle j,u\rangle$ & $(3,5)\times(1,3)$ & $\langle i,v\rangle$ \\
$\langle f,i\rangle$ & $(3,5)\times(1,4)$ & $0$ \\
$\langle f,j\rangle$ & $(3,5)\times(3,4)$ & $0$ \\
$\langle g,j\rangle$ & $(4,5)\times(1,3)$ & $\langle i,w\rangle$ \\
$\langle h,i\rangle$ & $(4,5)\times(1,4)$ & $0$ \\
$\langle h,j\rangle$ & $(4,5)\times(3,4)$ & $0$ \\
$\langle c,e\rangle$ & $(1,2)\times(2,4)$ & $\langle a,w\rangle$ \\
$\langle b,f\rangle$ & $(1,3)\times(3,4)$ & $\langle c,v\rangle$
\end{tabular}
\end{center}

Since we have 13 nonzero mutually orthogonal vectors in $\mathbb{R}^r$, we must have $r \ge 13$. Therefore $\operatorname{SOS}(P_G) = 13$, and $P_G$ is irreducible. Hence $z_2(5,4) \ge 13$.

By Proposition~\ref{prop:zwl54}, $z_{wL}(5,4) = 12$. Therefore,
\[
z_2(5,4) \ge 13 > 12 = z_{wL}(5,4).
\]
This completes the proof.
\end{proof}

\begin{remark}
Corollary~\ref{cor:z2-54} below upgrades this strict separation to the
exact identity $z_2(5,4)=13=z_{RL}(5,4)$.
\end{remark}

\begin{remark}
The construction violates $(W3)$. Consider the vertex-disjoint edges
\[
(1,2)\qquad\text{and}\qquad(4,3;2,4).
\]
Their cross cells are
\[
(1,3), (1,4), (4,2), (2,2)
\]
are occupied. Hence the weak admissibility test rejects the configuration, even though the coefficient identities prove that its SOS rank is 13.
\end{remark}

\section{\texorpdfstring{A Finite Intrinsic Replacement for (W3): $(RW3^+)$ and $z_{RL}$}{A Finite Intrinsic Replacement for (W3): (RW3+) and zRL}}

The $5 \times 4$ proof shows that the relevant object is not a literal empty-cross-cell test, but a finite proof system built from coefficient identities. Besides genuine rectangle identities, degenerate selected $2$-edges are also detected directly by same-row or same-column coefficient comparisons. We first package those mechanisms into the intrinsic fixed-point closure $(RW3)$, and then adjoin the complementary-pair rule to obtain the strengthened criterion $(RW3^+)$. This strengthened version is the one that defines the recursive-line parameter $z_{RL}$ studied throughout the section.

\begin{definition}[Line-and-rectangle certificate closure]
Let $G=(S,T,E_1\cup E_2)$ be a limited augmented bipartite graph, and let $\Omega$ be its occupied cells. For any unordered pair $\{p,q\}$ of occupied cells, write
\[
\delta(\{p,q\})=
\begin{cases}
1,&\text{if }\{p,q\}\in E_2,\\
0,&\text{otherwise.}
\end{cases}
\]
For a genuine rectangle with corners
\[
p=(i,j),\qquad q=(k,\ell),\qquad r=(i,\ell),\qquad s=(k,j),
\]
the same notation applies to the two diagonals $\{p,q\}$ and $\{r,s\}$.

Define an \emph{identification} relation $\sim$ on occupied cells as the least equivalence relation generated by the identification clauses in the rules below. Define a \emph{certified orthogonality} relation $\perp_R$ as the least symmetric relation on occupied cells that is closed under the orthogonality clauses below and saturated under $\sim$ in either argument. Concretely, $\sim$ and $\perp_R$ are generated by the following rules:
\begin{enumerate}
    \item \textbf{Line rule.} If $p,q\in \Omega$ lie in the same row or the same column, then the pair $\{p,q\}$ is certified to have value $\delta(\{p,q\})$:
    \begin{itemize}
        \item if $\delta(\{p,q\})=0$, then $p\perp_R q$;
        \item if $\delta(\{p,q\})=1$, then $p\sim q$.
    \end{itemize}

    \item \textbf{Saturation rule.} If $p\sim p'$, $q\sim q'$, and $p'\perp_R q'$, then $p\perp_R q$.

    \item \textbf{Rectangle transfer rule.} For any genuine rectangle with diagonals $\{p,q\}$ and $\{r,s\}$, suppose the diagonal $\{r,s\}$ is already certified to have value $\delta(\{r,s\})$ in the following sense:
    \begin{itemize}
        \item if $\delta(\{r,s\})=0$, then either one of $r,s$ is unoccupied or $r\perp_R s$;
        \item if $\delta(\{r,s\})=1$, then both $r,s$ are occupied and $r\sim s$.
    \end{itemize}
    Then the diagonal $\{p,q\}$ is certified to have value $\delta(\{p,q\})$:
    \begin{itemize}
        \item if $\delta(\{p,q\})=0$ and $p,q$ are occupied, then $p\perp_R q$;
        \item if $\delta(\{p,q\})=1$, then $p\sim q$.
    \end{itemize}
\end{enumerate}
\end{definition}

\begin{definition}[Recursive rectangle admissibility $(RW3)$]
For each selected edge $e$, let
\[
\operatorname{supp}(e)=
\begin{cases}
\{(i,j)\},&\text{if }e=(i,j)\in E_1,\\
\{(i,j),(k,\ell)\},&\text{if }e=(i,j;k,\ell)\in E_2.
\end{cases}
\]
We say that $G$ satisfies \emph{recursive rectangle admissibility} if, in the line-and-rectangle certificate closure above,
\begin{enumerate}
    \item every selected $2$-edge has its two halves identified;
    \item distinct selected edges determine distinct $\sim$-classes;
    \item for any two distinct selected edges $e,f$, there exist $p\in \operatorname{supp}(e)$ and $q\in \operatorname{supp}(f)$ such that $p\perp_R q$.
\end{enumerate}
\end{definition}

\begin{theorem}
If $G=(S,T,E_1\cup E_2)$ is recursively rectangle admissible, then the associated doubly simple biquadratic form $P_G$ is irreducible. In particular,
\[
\operatorname{SOS}(P_G)=|E_1|+|E_2|.
\]
\end{theorem}

\begin{proof}
Let
\[
P_G=\sum_{t=1}^r f_t^2,\qquad
f_t=\sum_{(i,j)\in \Omega} v_{ij}^{(t)}x_i y_j,
\]
be any SOS representation, and write $\mathbf v_{ij}=(v_{ij}^{(1)},\dots,v_{ij}^{(r)})^\top$. For occupied cells, $\|\mathbf v_{ij}\|=1$. We claim:
\begin{enumerate}
    \item if $p\sim q$, then $\mathbf v_p=\mathbf v_q$;
    \item if $p\perp_R q$, then $\mathbf v_p\cdot \mathbf v_q=0$.
\end{enumerate}
This is proved by induction on the closure depth. The saturation rule is immediate from equality of vectors.

For the line rule, if two occupied cells $p,q$ lie in one row or one column, then the coefficient of either $x_i^2y_jy_\ell$ or $x_ix_ky_j^2$ gives
\[
\mathbf v_p\cdot \mathbf v_q=\delta(\{p,q\}).
\]
Indeed, simplicity implies that the only way to obtain coefficient $2$ on that mixed term is for $\{p,q\}$ itself to be a selected degenerate $2$-edge.

For the rectangle transfer rule, the rectangle identity gives
\[
\mathbf v_p\cdot \mathbf v_q+\mathbf v_r\cdot \mathbf v_s
\;=\;
\delta(\{p,q\})+\delta(\{r,s\}).
\]
If the companion diagonal $\{r,s\}$ is certified to have value $\delta(\{r,s\})$, then the target diagonal $\{p,q\}$ has inner product $\delta(\{p,q\})$. Hence:
\begin{itemize}
    \item if $\delta(\{p,q\})=0$, then $\mathbf v_p\cdot \mathbf v_q=0$;
    \item if $\delta(\{p,q\})=1$, then $\mathbf v_p\cdot \mathbf v_q=1$, and since both vectors are unit vectors, $\mathbf v_p=\mathbf v_q$.
\end{itemize}
So the claim follows.

Now each selected edge determines one nonzero unit vector: a $1$-edge gives its unique occupied cell vector, and a $2$-edge gives the common vector of its two halves. By injectivity of the $\sim$-classes for selected edges, these are exactly $|E_1|+|E_2|$ distinct vectors. By the orthogonality condition in recursive rectangle admissibility, any two of them are orthogonal. Therefore $\mathbb R^r$ contains $|E_1|+|E_2|$ nonzero mutually orthogonal vectors, so
\[
r\ge |E_1|+|E_2|.
\]
The defining decomposition of $P_G$ already shows $\operatorname{SOS}(P_G)\le |E_1|+|E_2|$, hence equality holds and $P_G$ is irreducible.
\end{proof}

\begin{remark}
The recursive certificate $(RW3)$ should not be described as a strictly weaker version of the weak framework. The $5\times 4$ example above shows that $(RW3)$ can certify a configuration rejected by the literal weak cross-cell test $(W3)$, while the weak $6\times 3$ complementary-diagonal construction from \cite{qlc26} goes in the other direction: it is accepted by the weak argument, but its simultaneous complementary resolution is not captured by the present least-fixed-point rectangle-transfer rule alone. Thus the weak framework and $(RW3)$ alone are logically incomparable, which is why the final recursive-line parameter uses the strengthened closure $(RW3^+)$ rather than $(RW3)$ itself.
\end{remark}

\begin{definition}[Complementary-pair strengthening $(RW3^+)$]
Starting from the line-and-rectangle certificate closure, adjoin the following additional rule:
\begin{enumerate}
    \setcounter{enumi}{3}
    \item \textbf{Complementary-pair rule.} If $\{p,q\}\in E_2$ and $\{r,s\}\in E_2$ are the two diagonals of a genuine rectangle, then both selected diagonals are resolved simultaneously:
    \[
    p\sim q,\qquad r\sim s.
    \]
\end{enumerate}
We say that $G$ satisfies \emph{strengthened recursive rectangle admissibility} if the resulting closure satisfies the same three conditions as recursive rectangle admissibility. We denote this strengthened condition by $(RW3^+)$.
\end{definition}

\begin{proposition}
If $G=(S,T,E_1\cup E_2)$ satisfies $(RW3^+)$, then the associated doubly simple biquadratic form $P_G$ is irreducible. In particular,
\[
\operatorname{SOS}(P_G)=|E_1|+|E_2|.
\]
\end{proposition}

\begin{proof}
The proof is the same as for $(RW3)$ once one checks the new rule. If $\{p,q\}$ and $\{r,s\}$ are complementary selected diagonals of a rectangle, then the rectangle identity gives
\[
\mathbf v_p\cdot \mathbf v_q+\mathbf v_r\cdot \mathbf v_s=2.
\]
All four vectors are unit vectors, so each inner product is at most $1$. Therefore both inner products are equal to $1$, and hence $\mathbf v_p=\mathbf v_q$ and $\mathbf v_r=\mathbf v_s$. Thus the complementary-pair rule is sound, and the rest of the irreducibility argument is unchanged.
\end{proof}

\begin{proposition}\label{prop:weak-implies-rw3plus}
Every weak admissible limited augmented bipartite graph satisfies the strengthened condition $(RW3^+)$.
\end{proposition}

\begin{proof}
{
Let $G=(S,T,E_1\cup E_2)$ be weak admissible. We verify the three defining conditions of $(RW3^+)$.

First, every selected $2$-edge is resolved. Row-degenerate and column-degenerate $2$-edges are resolved by the line rule, and every mutually complementary pair of nondegenerate selected $2$-edges is resolved by the complementary-pair rule. Contract these complementary pairs in the weak $(W2)$ dependency graph. By weak admissibility the contracted graph is acyclic, so we may process its vertices in reverse topological order.

Let $e=(i,j;k,\ell)$ be the current nondegenerate selected $2$-edge and write $c=(i,\ell)$, $d=(k,j)$ for the opposite diagonal. If one of $c,d$ is unoccupied, then one rectangle transfer resolves $e$. Otherwise both are occupied. By $(W2')$ they are not both 1-edges. Every selected 2-edge occupying one of them is either degenerate, hence already resolved, or an outgoing dependency of $e$, hence already resolved by the reverse-topological induction. Let $f,g$ be the selected edges containing $c,d$. If resolved representatives of $f,g$ share a row or a column, then the line rule and saturation give $c\perp_R d$. Otherwise $f,g$ are vertex-disjoint and at least one is a 2-edge, so $(W3)$ and Lemma~\ref{lem:empty-cross-cell} give $c\perp_R d$. The rectangle of $e$ then resolves $e$.

Second, distinct selected edges remain in distinct equivalence classes, because every value-one identification introduced by the rules joins only the two halves of one selected 2-edge. Simplicity keeps supports of distinct selected edges disjoint, so different selected edges cannot collapse into one equivalence class.

Third, let $e\neq f$ be selected edges. If representatives of $e$ and $f$ share a row or column, then the line rule gives the required orthogonality. Otherwise, if at least one of $e,f$ is a 2-edge, then $e$ and $f$ are vertex-disjoint and their opposite-cell set contains an unoccupied cell by weak admissibility, so Lemma~\ref{lem:empty-cross-cell} gives the orthogonality certificate. Finally consider two 1-edges $e=(i,j)$ and $f=(k,\ell)$, with cross cells $c=(i,\ell)$ and $d=(k,j)$. If one of $c,d$ is unoccupied, Lemma~\ref{lem:empty-cross-cell} applies directly. If $c,d$ are the two halves of one resolved selected $2$-edge, then the rectangle on rows $i,k$ and columns $j,\ell$ has companion diagonal of value one and therefore gives $e\perp_R f$. Otherwise let $g,h$ be the distinct resolved selected edges containing $c,d$. By $(W1)$ they are not both 1-edges. If representatives of $g,h$ share a row or column, then the line rule and saturation give $c\perp_R d$; otherwise $g,h$ are vertex-disjoint and at least one is a 2-edge, so $(W3)$ and Lemma~\ref{lem:empty-cross-cell} again give $c\perp_R d$. A final rectangle transfer yields $e\perp_R f$.
}
\end{proof}

\begin{lemma}\label{lem:empty-cross-cell}
{
Suppose two distinct selected edges $e,f$ are already resolved, where a 1-edge is regarded as resolved automatically. If $e$ and $f$ are vertex-disjoint and their opposite-cell set contains an unoccupied cell, then the line-and-rectangle closure certifies orthogonality of representatives of $e$ and $f$.
}
\end{lemma}

\begin{proof}
{
Suppose the empty cross cell is $(r,c)$, where row $r$ occurs in the support of $e$ and column $c$ occurs in the support of $f$; the other case is symmetric. Choose a support cell $p=(r,c_e)$ of $e$ in row $r$ and a support cell $q=(r_f,c)$ of $f$ in column $c$. Vertex-disjointness gives $r\neq r_f$ and $c_e\neq c$, so
\[
p,\ q,\ (r,c),\ (r_f,c_e)
\]
form a genuine rectangle. The companion diagonal contains the unoccupied cell $(r,c)$ and is therefore certified zero. Since simplicity prevents $p$ and $q$ from forming a third selected 2-edge, the target diagonal has value $0$, so one rectangle transfer gives $p\perp_R q$. Saturation through the resolved halves gives the desired edge orthogonality.
}
\end{proof}

\begin{corollary}\label{cor:rw3plus-strictly-weaker}
On limited augmented bipartite graphs satisfying $(S)$ and excluding $(W1)$, $(W2)$, and $(W2')$, the strengthened certificate $(RW3^+)$ is strictly weaker than the literal weak cross-cell test $(W3)$.
\end{corollary}

\begin{proof}
If a graph excludes the forbidden weak cross-cell pattern $(W3)$ together with the common weak conditions, then it is weak admissible, so Proposition~\ref{prop:weak-implies-rw3plus} gives $(RW3^+)$. The inclusion is strict because Proposition~\ref{prop:rw3-not-w3-54} gives an explicit $5\times 4$ graph satisfying $(RW3)$, hence also $(RW3^+)$, while failing the literal weak cross-cell test $(W3)$.
\end{proof}

\begin{remark}
The strengthened certificate $(RW3^+)$ is therefore a single-branch refinement: it adds complementary-pair checking to the intrinsic recursive closure and becomes strictly weaker than the literal weak cross-cell test $(W3)$. Corollary~\ref{cor:rw3plus-strictly-weaker} makes the separate literal $(W3)$ branch redundant on the weak-admissible class, so the final recursive-line parameter is built from $(RW3^+)$ itself.
\end{remark}

\begin{definition}[Recursive-line Zarankiewicz number]
Let $z_{RL}(m,n)$ be the maximum total number of edges $|E_1|+|E_2|$ over all limited augmented bipartite graphs with $|E_1|=z(m,n)$ that satisfy $(S)$, whose $1$-edge graph $G_1$ is $C_4$-free, and that satisfy strengthened recursive rectangle admissibility $(RW3^+)$.
\end{definition}

\begin{proposition}\label{prop:zrl-lower}
For all $m,n\ge 2$,
\[
z_2(m,n)\ge z_{RL}(m,n)\ge z_{wL}(m,n).
\]
\end{proposition}

\begin{proof}
Let $G$ be counted by $z_{RL}(m,n)$. Then $G$ satisfies $(RW3^+)$, so the proposition above shows that its associated doubly simple biquadratic form is irreducible. Hence $G$ contributes a graph counted by $z_2(m,n)$, and therefore $z_2(m,n)\ge z_{RL}(m,n)$.

Moreover, every weak admissible graph satisfies $(RW3^+)$ by Proposition~\ref{prop:weak-implies-rw3plus}. Hence every graph counted by $z_{wL}(m,n)$ is also counted by $z_{RL}(m,n)$, giving $z_{RL}(m,n)\ge z_{wL}(m,n)$.
\end{proof}

We now record the first exact family for the recursive-line parameter. Combined with the three-column upper bound from Section~3, it yields an infinite exact family simultaneously for $z_{RL}$ and $z_2$.

\subsection{The exact three-column family}

\begin{proposition}[Three-column chain lower bound]\label{prop:zrl-m3-chain}
For every $m\ge 3$,
\[
z_{RL}(m,3)\ge 2m.
\]
Consequently,
\[
z_2(m,3)=z_{RL}(m,3)=2m.
\]
\end{proposition}

\begin{proof}
For every $m\ge 3$, set
\[
E_1=\{(1,1),(1,2),(2,1),(2,3),(3,2),(3,3)\},\qquad E_2=\emptyset.
\]
If $m\ge 4$, enlarge this by adjoining the singleton edges $(i,2)$ for $4\le i\le m$ and the chain of $2$-edges
\[
E_2=\{(i,1;i+1,3):3\le i\le m-1\}.
\]
Then $|E_1|=m+3=z(m,3)$, $|E_2|=m-3$, and therefore the total number of selected edges is $2m$. The only unoccupied cells are
\[
(1,3),\qquad (2,2),\qquad (m,1).
\]
The column supports of $E_1$ are $C_1=\{1,2\}$, $C_2=\{1,3,4,\dots,m\}$, and $C_3=\{2,3\}$, so $G_1$ is $C_4$-free.

It remains to check $(RW3^+)$. If $m=3$, every selected edge is a $1$-edge. Pairs on a common line are handled by the line rule, and otherwise the companion diagonal of the rectangle contains a hole because $G_1$ is $C_4$-free. Thus $(RW3^+)$ holds when $m=3$.

Assume now $m\ge 4$ and write $e_i=(i,1;i+1,3)$ for $3\le i\le m-1$. Resolve these selected $2$-edges from right to left. The terminal edge $e_{m-1}$ is resolved by the rectangle on rows $m-1,m$ and columns $1,3$, whose companion diagonal contains the hole $(m,1)$. If $e_{i+1}$ has been resolved, then $(i+1,1)\sim(i+2,3)$. Since $(i,3)$ and $(i+2,3)$ share column $3$ and belong to distinct selected edges, the line rule and saturation give $(i,3)\perp_R(i+1,1)$, so the rectangle on rows $i,i+1$ and columns $1,3$ resolves $e_i$. Thus all selected $2$-edges are resolved.

For $m\ge 4$, no complementary selected pair occurs, and the only pairs with $\delta=1$ are the chosen diagonals in $E_2$, so distinct selected edges remain in distinct $\sim$-classes. Let $a_i$ be the common vector of $e_i$. Distinct chain edges are orthogonal by column $3$. One-edges in columns $1$ and $3$ are orthogonal to every $a_i$ by the line rule, and $(1,2)$ is orthogonal to every $a_i$ by the rectangle on rows $1,i+1$ and columns $2,3$, whose companion diagonal contains the hole $(1,3)$.

For $k=3,\dots,m$, let $u_k$ be the vector at $(k,2)$ and put $g_{k,i}=\langle u_k,a_i\rangle$ for $3\le i\le m-1$. Same-row orthogonality gives the grounded values
\[
g_{i,i}=g_{i+1,i}=0.
\]
If $k=i-1$, the rectangle on rows $(i-1,i)$ and columns $(1,2)$ gives $g_{i-1,i}=-g_{i,i-1}=0$. If $k\le i-2$, the rectangles on rows $(k,i)$ with columns $(1,2)$ and on rows $(k+1,i)$ with columns $(2,3)$ give $g_{k,i}=g_{k+1,i-1}$, so repeated transport reaches a grounded adjacent value. If $k\ge i+2$, the rectangle on rows $i+1,k$ and columns $2,3$ gives $g_{k,i}=-g_{i+1,k-1}$, and the right-hand side has already been settled. Therefore $g_{k,i}=0$ for all $k,i$, so every one-edge vector is orthogonal to every chain vector $a_i$.

Finally take two distinct selected $1$-edges $p,q$. If they share a line, use the line rule. Otherwise let $r,s$ be the companion diagonal of their rectangle. Since $G_1$ is $C_4$-free, $r,s$ cannot both be $1$-edges. The remaining possibilities are: one of $r,s$ is a hole; $r,s$ lie in two distinct chain edges; $r,s$ are the two halves of one chain edge; or one is a $1$-edge and the other lies in a chain edge. Each case was already settled above, so the rectangle identity gives $p\perp_R q$. Therefore the closure satisfies $(RW3^+)$, so this family is counted by $z_{RL}(m,3)$ and
\[
z_{RL}(m,3)\ge |E_1|+|E_2|=2m.
\]
Finally Proposition~\ref{prop:zrl-lower} gives $z_2(m,3)\ge z_{RL}(m,3)$,
while Theorem~\ref{thm:m3-upper} gives $z_2(m,3)\le2m$, so equality holds
throughout.
\end{proof}

\begin{corollary}[Three-column separation]\label{cor:three-column-gap}
For every $m\ge10$,
\[
z_2(m,3)=z_{RL}(m,3)>z_{wL}(m,3).
\]
For $m\ge16$, the gap is exactly
\[
z_2(m,3)-z_{wL}(m,3)
=z_{RL}(m,3)-z_{wL}(m,3)
=\left\lfloor\frac{m-3}{3}\right\rfloor.
\]
In particular, the gap divided by $m$ tends to $1/3$.
\end{corollary}
\begin{proof}
Proposition~\ref{prop:zrl-m3-chain} gives $z_2(m,3)=z_{RL}(m,3)=2m$.
For $m\ge16$, the formula in \cite{qlc26} is
\[
z_{wL}(m,3)=m+3+\left\lfloor\frac{2m-4}{3}\right\rfloor,
\]
and subtraction gives the stated gap. For the remaining sizes, the exact
weak values in \cite{qlc26} give
\[
\begin{array}{c|rrrrrr}
m&10&11&12&13&14&15\\ \hline
z_2(m,3)=z_{RL}(m,3)&20&22&24&26&28&30\\
z_{wL}(m,3)&19&21&22&24&25&27\\
\text{Gap}&1&1&2&2&3&3
\end{array}
\]
and hence strict separation throughout $m\ge10$.
\end{proof}

\begin{remark}
The inclusion $z_{RL}(m,n)\ge z_{wL}(m,n)$ follows from Proposition~\ref{prop:zrl-lower}, and the corollary above shows that the gap can already be linear in the three-column family. We next turn to the smallest four-column case, where exact search shows that the recursive-line parameter is again strictly larger than the weak one. At $(5,4)$, exact search under the revised definition gives
\[
z_{RL}(5,4)=13>12=z_{wL}(5,4),
\]
and Theorem~\ref{thm:rw3-77} further gives
\[
z_{RL}(7,7)\ge 32 > 28 = z_{wL}(7,7).
\]
\end{remark}

\begin{theorem}\label{prop:zrl54}
One has
\[
z_{RL}(5,4)=13.
\]
\end{theorem}

\begin{proof}
The explicit $5\times 4$ configuration constructed earlier has $|E_1|=10$ and $|E_2|=3$, hence total size $13$, and Proposition~\ref{prop:rw3-not-w3-54} below shows that it satisfies $(RW3)$. Therefore it is counted by $z_{RL}(5,4)$, so $z_{RL}(5,4)\ge 13$.

For the reverse inequality, enumerate the extremal $C_4$-free
$5\times4$ bases with ten $1$-edges and all disjoint families of pairs on
their remaining cells. There are $2640$ labeled bases in three row-column
isomorphism classes. Over one representative of each class, the search
finds no $(RW3^+)$-admissible total-$15$ or total-$14$ family and $124$
admissible total-$13$ families. The last count is taken before quotienting
the augmentations by base automorphisms; the three representatives
contribute $50$, $26$, and $48$ families. Both independent exact closure
implementations reproduce these counts; the bases, candidate records, and
witnesses are archived in \cite{reproducibility}.
The universal cell bound is $15$, so these exclusions give
$z_{RL}(5,4)\le13$.
\end{proof}

\begin{lemma}[Hereditary irreducibility]\label{lem:hereditary}
If a displayed irreducible SOS decomposition is given and a collection of
complete displayed squares is deleted, then the remaining displayed
decomposition is still irreducible.
\end{lemma}

\begin{proof}
If the remainder admitted a shorter representation, inserting the deleted
squares back would shorten the original representation.
\end{proof}

\begin{lemma}[Base classification]\label{lem:base-ABC}
One has $z(5,4)=10$. Up to row and column permutation, the extremal
ordinary $5\times4$ supports are exactly the three bases
\begin{equation}\label{eq:ABC-bases}
\begin{aligned}
A&:\ s{:}a,\ p{:}ab,\ q{:}ac,\ r{:}ad,\ t{:}bcd,\\
B&:\ s{:}a,\ p{:}ab,\ q{:}bc,\ r{:}bd,\ t{:}acd,\\
C&:\ s{:}ab,\ p{:}ac,\ q{:}bc,\ r{:}ad,\ t{:}bd.
\end{aligned}
\end{equation}
Here, for example, $p{:}ab$ means that $(pa)^2$ and $(pb)^2$ are ordinary
squares.
\end{lemma}

\begin{proof}
Write $d_i$ for the ordinary row degrees. Since $G_1$ is $C_4$-free,
$\sum_i\binom{d_i}{2}\le\binom{4}{2}=6$. The integer inequality
$\binom{d}{2}\ge 2d-3$ therefore yields $2|E_1|-15\le6$, hence
$|E_1|\le10$. The displayed bases attain ten. At ten one has
$\sum_i d_i=10$ and
\[
\sum_i\binom{d_i}{2}=5+\frac12\sum_i(d_i-2)^2\le6.
\]
Thus the degrees are either all two or $(3,2,2,2,1)$. In the first case
the five row neighborhoods are five distinct edges of $K_4$, giving $C$.
In the second, the degree-three row uses a triple of columns. Each
degree-two row must join the remaining column to a distinct member of
that triple. The singleton lies either outside or inside the triple,
giving $A$ or $B$.
\end{proof}

\begin{corollary}\label{cor:z2-54}
One has
\[
z_2(5,4)=z_{SL}(5,4)=z_{RL}(5,4)=13.
\]
\end{corollary}

\begin{proof}
The explicit irreducible example from Section~3 gives $z_2(5,4)\ge13$, and
Theorem~\ref{prop:zrl54} gives $z_{RL}(5,4)=13$. By
Proposition~\ref{prop:zsl-hierarchy} one has $z_2\ge z_{SL}\ge z_{RL}$, so
it remains only to prove the matching unrestricted upper bound
$z_2(5,4)\le13$.

Let $G$ be an irreducible $5\times4$ doubly simple augmentation with
$|E_1|=z(5,4)=10$. If $|E_2|\ge5$, then deleting one displayed $2$-edge
square leaves, by Lemma~\ref{lem:hereditary}, an irreducible total-$14$
augmentation. Thus it is enough to exclude irreducible total-$14$
augmentations, i.e., the case $|E_2|=4$.

By Lemma~\ref{lem:base-ABC}, the extremal ordinary $5\times4$ supports
are, up to row and column permutation, the three bases
in~\eqref{eq:ABC-bases}.

For base $A$, the free cells are $s{:}bcd$, $p{:}cd$, $q{:}bd$, $r{:}bc$,
and $t{:}a$. The hole-pressure and strip-overload identities of
\hyperref[app:z254]{Appendix~B} force $ta$ to be a hole and leave only six possible core pairs
\[
pc+qd,\ pd+qb,\ pc+rb,\ pd+rc,\ qb+rc,\ qd+rb,
\]
all equivalent under simultaneous permutations of $(p,b)$, $(q,c)$, and
$(r,d)$. At most one core pair can be selected, and the remaining three
pairs are then forced to touch the distinguished row $s$. Each of the two
surviving final patterns is excluded by an explicit product relation.
Hence base $A$ admits no irreducible augmentation with four selected
pairs. The identities are recorded in \hyperref[app:z254]{Appendix~B}.

For base $B$, after the explicit one-pair local exclusions of \hyperref[app:z254]{Appendix~B}
only $20$ selected-pair positions remain. They lie in six forbidden
classes $B_1,\dots,B_6$, each containing at most one selected pair, with
weights $1,1,1,1,1,2$. The weighted class count is exactly
\[
2|E_2|+\xi_5+\xi_6+\xi_{19}+\xi_{20}\le 7,
\]
where the $\xi_i$ are the selection indicators of the remaining pair
positions. In particular $2|E_2|\le7$, contradicting $|E_2|=4$.

For base $C$, the explicit one-pair exclusions leave $25$ possible
selected-pair positions. They lie in eleven forbidden classes
$C_1,\dots,C_{11}$, each containing at most one selected pair, and the
corresponding weighted count is
\[
3|E_2|+2\eta_1\le 11,
\]
where the $\eta_i$ are the associated selection indicators. Hence
$3|E_2|\le11$, again contradicting $|E_2|=4$.

The surviving pair lists, the definitions of the classes $B_i$ and $C_i$,
and the local certificates proving that every listed class contains at
most one selected pair are given in \hyperref[app:z254]{Appendix~B}. Therefore no irreducible
total-$14$ augmentation exists, and hence no irreducible total-$15$
augmentation exists either. So $z_2(5,4)\le13$. Therefore
\[
z_2(5,4)=13=z_{SL}(5,4)=z_{RL}(5,4).
\]
\end{proof}

\begin{proposition}\label{prop:rw3-not-w3-54}
The explicit $5\times 4$ configuration from the gap theorem above satisfies both $(RW3)$ and $(RW3^+)$, but violates literal $(W3)$.
\end{proposition}

\begin{proof}
The initial identifications
\[
\mathbf v_{31}=\mathbf v_{42},\qquad
\mathbf v_{52}=\mathbf v_{33},\qquad
\mathbf v_{43}=\mathbf v_{24},
\]
together with the $40$-line rectangle table in the gap proof above, form a finite line-and-rectangle certificate proving recursive rectangle admissibility. Therefore the graph satisfies $(RW3)$, and hence also the strengthened condition $(RW3^+)$. The failure of literal $(W3)$ was recorded in the remark immediately preceding Section~4. Hence this example shows that even the strengthened recursive certificate can hold when literal $(W3)$ fails.
\end{proof}

\subsection{A complementary-pair example}

After the chain family and the exact $(5,4)$ computation, it is useful to
see a different kind of optimizer that genuinely uses the complementary-pair
rule. The following example illustrates that strengthening on a $13\times 3$
grid.

\begin{example}[Complementary resolution on a $13\times3$ grid]
\label{ex:complementary13}
Let
\[
\begin{aligned}
E_1=\{&(1,1),(1,2),(2,1),(2,3),(3,2),(3,3),\\
&(4,1),(5,1),(6,1),(7,1),(8,1),\\
&(9,2),(10,2),(11,2),(12,2),(13,3)\}
\end{aligned}
\]
and
\[
\begin{aligned}
E_2=\{&(1,3;6,2),\ (2,2;5,3),\ (3,1;9,3),\ (4,2;4,3),\\
&(5,2;8,3),\ (6,3;7,2),\ (9,1;10,3),\\
&(11,1;12,3),\ (11,3;12,1),\ (13,1;13,2)\}.
\end{aligned}
\]
The supports are simple, $E_1$ is $C_4$-free, and
$|E_1|=16=z(13,3)$, $|E_2|=10$.

The selected diagonals
\[
(11,1;12,3),\qquad(11,3;12,1)
\]
form a complementary pair. The base $(RW3)$ closure cannot initiate either
identification: each rectangle transfer requires the other selected
diagonal to be resolved first, and none of these four cells belongs to
another selected edge. The complementary-pair rule resolves them
simultaneously. The resulting $(RW3^+)$ closure resolves all ten selected
$2$-edges and certifies all $\binom{26}{2}=325$ orthogonality obligations.
Both independent exact checkers verify this certificate; the witness and
verification records are archived in \cite{reproducibility}.
Thus this configuration attains the value $26$ already established by
Proposition~\ref{prop:zrl-m3-chain}.
\end{example}

Additional certified constructions at $m=10,14,15$ are retained in the
repository. Their numerical optimality follows from the same general
three-column theorem.

\subsection{Further exact finite values}

\begin{proposition}[Exact values on small grids]\label{prop:finite-values}
One has
\[
\begin{aligned}
z_{RL}(4,4)&=10,& z_{RL}(6,4)&=16,\\
z_{RL}(5,5)&=17,& z_{RL}(7,4)&=19.
\end{aligned}
\]
\end{proposition}
\begin{proof}
The classical extrema are $z(4,4)=9$, $z(6,4)=12$, $z(5,5)=12$,
and $z(7,4)=13$. Up to row and column permutations, there is one
base in each case except $(5,5)$, where there are two. The base
classification and the exhaustive augmentation searches are archived in
\cite{reproducibility}.

For a base with $f$ free cells, the number of families of $k$ disjoint
unordered pairs is
\[
M(f,k)=\frac{f!}{(f-2k)!\,2^k k!}.
\]
Enumerating these families and testing $(RW3^+)$ gives the following
counts. They are over one representative of each base, before quotienting
augmentations by base automorphisms. Both exact closure implementations
give the same acceptance and rejection records.
\begin{center}
\begin{tabular}{ccrr}
\toprule
Grid & Total & Candidates & Accepted\\
\midrule
$4\times4$ & 10 & 21 & 6\\
$4\times4$ & 11 & 105 & 0\\
$4\times4$ & 12 & 105 & 0\\
$6\times4$ & 16 & 51\,975 & 6\\
$6\times4$ & 17 & 62\,370 & 0\\
$6\times4$ & 18 & 10\,395 & 0\\
$5\times5$ & 17 & 540\,540 & 4\\
$5\times5$ & 18 & 270\,270 & 0\\
$7\times4$ & 19 & 4\,729\,725 & 18\\
$7\times4$ & 20 & 2\,027\,025 & 0\\
\bottomrule
\end{tabular}
\end{center}
Positive counts supply explicit witnesses. The zero counts exclude every
larger total up to the universal cell bound, proving the four equalities.
The corresponding exact weak values are reproduced by the separate
weak-admissibility searches in the same archive.
\end{proof}

\begin{corollary}\label{cor:z2-64}
One has
\[
z_2(6,4)=z_{SL}(6,4)=z_{RL}(6,4)=16.
\]
\end{corollary}

\begin{proof}
Proposition~\ref{prop:finite-values} gives $z_{RL}(6,4)=16$, hence
$z_2(6,4)\ge16$, and Proposition~\ref{prop:zsl-hierarchy} then also gives
$z_{SL}(6,4)=16$. For the reverse inequality, let an irreducible
$6\times4$ augmentation have $12+s$ displayed squares, where $12=z(6,4)$.
Delete one row at a time and keep only the displayed squares whose full
support remains in the other five rows. Equivalently, one deletes every
complete displayed square touching the chosen row. By
Lemma~\ref{lem:hereditary}, each remaining $5\times4$ restriction is
still irreducible, and is therefore an augmentation of an extremal
$5\times4$ base. Corollary~\ref{cor:z2-54} then bounds every such
restriction by $13$.

Summing over the six deleted rows, each ordinary square is counted five
times. A non-row-degenerate selected pair is counted four times, while a
row-degenerate selected pair is counted five times. Writing $\ell$ for the
number of row-degenerate selected pairs, we obtain
\[
5\cdot 12 + 4(s-\ell)+5\ell \le 6\cdot 13,
\]
so $4s+\ell\le18$. In particular $s\le4$, and therefore
\[
z_2(6,4)\le 12+4=16.
\]
Hence $z_2(6,4)=16=z_{SL}(6,4)=z_{RL}(6,4)$.
\end{proof}

\section{A Computational \texorpdfstring{$(RW3)$}{(RW3)} Benchmark for the 7-by-7 Case}
\label{sec:benchmark77}

Throughout this section, rows and columns are numbered by $\mathbb Z_7$.
Fix the extremal $C_4$-free base
\[
E_1=\{(r,r+d):r\in\mathbb Z_7,\ d\in\{0,1,3\}\},
\]
with column arithmetic modulo seven. It has $21=z(7,7)$ cells.
The benchmark uses the following specified weak-optimal augmentation:
\[
E_2^{\mathrm{wk}}=\left\{\begin{gathered}
(0,2;1,5),\ (0,6;2,1),\ (1,0;6,3),\\
(2,4;5,3),\ (3,2;4,6),\ (3,5;5,0),\\
(4,1;6,4)
\end{gathered}\right\}.
\]

An improved recursive augmentation on the same base is
\[
E_2^{32}=\left\{\begin{gathered}
(0,2;6,4),\ (0,4;4,6),\ (0,5;1,3),\\
(1,0;3,2),\ (1,5;2,6),\ (1,6;5,0),\\
(2,0;3,5),\ (2,1;6,3),\ (2,4;5,3),\\
(3,1;4,3),\ (4,1;6,5)
\end{gathered}\right\}.
\]

\begin{theorem}\label{thm:rw3-77}
For the base and augmentations above:
\begin{enumerate}
    \item $E_1\cup E_2^{\mathrm{wk}}$ is weak admissible and satisfies
    $(RW3)$. Its total size is $28=z_{wL}(7,7)$.
    \item Of the $91$ disjoint extra $2$-edges that can be added to this
    particular weak augmentation, exactly seven pass the filter consisting
    of $(S)$, exclusion of $(W2)$ and $(W2')$, and $(RW3)$. They are
\[
\mathcal E_{\mathrm{extra}}=\left\{\begin{gathered}
(0,4;1,6),\ (0,4;2,0),\ (1,6;6,5),\\
(2,0;5,2),\ (3,1;4,3),\ (3,1;5,2),\\
(4,3;6,5)
\end{gathered}\right\}.
\]

    Each therefore gives a certified total-$29$ augmentation.
    \item $E_1\cup E_2^{32}$ passes the same filter and has $21+11=32$
    selected edges. Consequently,
    \[
    z_2(7,7)\ge z_{RL}(7,7)\ge32>28=z_{wL}(7,7).
    \]
\end{enumerate}
\end{theorem}
\begin{proof}
The finite classification in \cite{reproducibility} verifies that every
extremal $7\times7$ base is row-column isomorphic to the displayed base.
There are $168$ individually weak-admissible candidate $2$-edges on its
free cells. The $168$ archived automorphisms act transitively on these
candidates, so every nonempty weak augmentation can be normalized to
contain $((0,2),(1,5))$. Exhaustive enumeration with this pair fixed gives
the following numbers of accepted partial families of sizes $1,\ldots,8$:
\[
1,\quad98,\quad2305,\quad12715,\quad10920,\quad681,\quad2,\quad0.
\]
Two independent implementations give the same counts.
Weak admissibility is preserved by deleting selected $2$-edges, so the
absence of a size-eight family excludes every larger weak augmentation.
The displayed size-seven family therefore proves $z_{wL}(7,7)=28$.
Both exact line-and-rectangle checkers also accept this family under
$(RW3)$.

The weak augmentation leaves $14$ cells unused, giving
$\binom{14}{2}=91$ possible extra pairs. Testing all of them against the
stated filter gives exactly the seven displayed extensions. The complete
decision record is archived in \cite{reproducibility}.

For the total-$32$ family, both independent $(RW3)$ implementations resolve
all eleven selected $2$-edges and certify all $\binom{32}{2}=496$ pairwise
orthogonality obligations. Simplicity and the two dependency restrictions
are checked separately. The witness and its verification record are
archived in \cite{reproducibility}.
Soundness of $(RW3)$ gives the lower bound of $32$. This is a lower bound
on $z_{RL}(7,7)$; optimality of $32$ is not asserted.
\end{proof}

\section{The Complete-Graph Incidence Family: A Cubic Separation}\label{sec:incidence}

We now show that the gap between $z_2$ and $z_{wL}$ also occurs in an infinite family and grows cubically.

\subsection{The Incidence Graph}

Let the columns be the vertices of $K_N$ and the rows its $\binom{N}{2}$ edges.
The one-edge set is the incidence graph:
\[
E_1 = \{(ab,a),(ab,b): a < b\}.
\]
Thus
\[
m = \binom{N}{2}, \qquad n = N, \qquad |E_1| = N(N-1).
\]

\begin{theorem}\label{thm:incidence-unique}
For $m = \binom{N}{2}$, $n = N$,
\[
z\left(\binom{N}{2}, N\right) = N(N-1),
\]
and every extremal $C_4$-free graph is the incidence graph of $K_N$ up to row
and column relabeling.
\end{theorem}

\begin{proof}
{
If the row degrees are $d_r$, then $C_4$-freeness implies
\[
\sum_r \binom{d_r}{2} \le \binom{N}{2} = m.
\]
Since $\binom{d}{2} \ge d-1$ for every $d\ge 0$,
\[
|E_1|-m=\sum_r (d_r-1)\le \sum_r \binom{d_r}{2}\le m,
\]
so $|E_1| \le 2m = N(N-1)$. If equality holds, then equality holds throughout the displayed chain. Hence every row degree lies in $\{1,2\}$, and in particular no row has degree zero. Since there are exactly $m$ rows and
\[
\sum_r \binom{d_r}{2}=m,
\]
every row must in fact have degree $2$. Equality in the pair count then says that every pair of columns occurs in exactly one row, which is precisely the incidence graph of $K_N$ up to row and column relabeling.
}
\end{proof}

\subsection{A Weak Upper Bound for the Incidence Family}

Let $F = \frac{N(N-1)(N-2)}{2}$ be the number of nonincidence cells, and let
$s = |E_2|$. Let $Z = F - 2s$ be the number of unused nonincidence cells.

Every selected 2-edge in the incidence graph has at least
\[
A_N = (N-1)(N-2) - 4
\]
vertex-disjoint incidence 1-edges. If a hole is at $(r,c)$, let $r_r$ be the
number of selected 2-edges using row $r$ and $c_c$ the number using column $c$.
That hole can cover at most
\[
(N-1)r_r + 2c_c
\]
of the counted $(W3)$ pairs. If $z_r$ and $z_c$ denote the number of holes in
row $r$ and column $c$, simplicity gives
\[
r_r \le N-2-z_r, \qquad c_c \le \binom{N-1}{2} - z_c.
\]
{
\[
\begin{aligned}
sA_N
&\le (N-1)\sum_r z_r r_r + 2\sum_c z_c c_c \\
&\le (N-1)\sum_r z_r(N-2-z_r) + 2\sum_c z_c\left(\binom{N-1}{2}-z_c\right) \\
&= 2(N-1)(N-2)Z - (N-1)\sum_r z_r^2 - 2\sum_c z_c^2.
\end{aligned}
\]
Now Cauchy--Schwarz over the $\binom{N}{2}$ rows and the $N$ columns gives
\[
\sum_r z_r^2 \ge \frac{Z^2}{\binom{N}{2}},\qquad \sum_c z_c^2 \ge \frac{Z^2}{N},
\]
so
\[
sA_N \le 2(N-1)(N-2)Z - \frac{4Z^2}{N}.
\]
Since $s = (F-Z)/2$, this factors to give
\[
\frac{Z}{F} \ge \frac{1}{4} - \frac{1}{(N-1)(N-2)}.
\]
Equivalently,
\[
s \le \frac{3}{8}F + \frac{N}{4}.
\]
Degenerate 2-edges cause no difficulty here: the row and column exclusions above are upper estimates and only become more conservative for degenerate supports.
}

\begin{corollary}\label{cor:incidence-weak-upper}
\[
z_{wL}\left(\binom{N}{2}, N\right) \le N(N-1) + \frac{3N(N-1)(N-2)}{16} + \frac{N}{4}.
\]
In particular,
\[
z_{wL}\left(\binom{N}{2}, N\right) \le \left(\frac{3}{16} + o(1)\right)N^3.
\]
\end{corollary}

\subsection{The Nested Perfect-One-Factorization Construction}

Let $N = 2p$ with $p$ an odd prime. Choose an outer perfect one-factorization
$\mathcal{F}$ of $K_{2p}$, whose existence is classical \cite{kobayashi1989}.
Inside each outer factor $F$ of size $p$, label its row-edges by $\mathbb{F}_p$.
For an edge $g \in F$ with label $a$, pair the remaining row labels by the
involution
\[
a+x \longleftrightarrow a-x, \qquad x = 1,\ldots,(p-1)/2,
\]
and on every resulting pair of rows use the two endpoints of $g$ as columns
and select both complementary diagonals.

Every nonincidence cell is used exactly once. Hence
\[
|E_2| = p(2p-1)(p-1), \qquad |E_1| + |E_2| = p(2p-1)(p+1).
\]

\begin{theorem}\label{thm:nested-p1f}
For every odd prime $p$,
\[
z_2\left(\binom{2p}{2}, 2p\right) \ge p(2p-1)(p+1).
\]
Moreover, Proposition~\ref{prop:universal-cell-bound} gives equality:
\[
z_2\left(\binom{2p}{2}, 2p\right) = p(2p-1)(p+1).
\]
\end{theorem}

\begin{proof}
The complete coefficient-vector proof is given in Appendix~A. The lower bound there shows that every SOS representation has at least $p(2p-1)(p+1)$ squares, while the displayed construction has exactly that many. Equality then follows from Proposition~\ref{prop:universal-cell-bound}, because here
\[
m=\binom{2p}{2}=p(2p-1),\qquad n=2p,\qquad z(m,n)=2p(2p-1),
\]
so
\[
\left\lfloor \frac{mn+z(m,n)}{2}\right\rfloor
=\left\lfloor \frac{2p^2(2p-1)+2p(2p-1)}{2}\right\rfloor
=p(2p-1)(p+1).
\]
\end{proof}

\begin{corollary}\label{cor:nested-rw3plus}
For every odd prime $p\ge 5$,
\[
z_{RL}\left(\binom{2p}{2},2p\right)=z_2\left(\binom{2p}{2},2p\right)=p(2p-1)(p+1).
\]
\end{corollary}

\begin{proof}
Appendix~A resolves every selected complementary pair by the complementary-pair rule and then proves all remaining orthogonality relations for $p\ge 5$ by finite chains of ordinary rectangle transfers terminating at grounded same-row or same-column zeros. Reading those chains backwards gives a valid $(RW3^+)$ certificate for the same construction. Thus the construction is counted by $z_{RL}$, while Theorem~\ref{thm:nested-p1f} gives the matching $z_2$ value.
\end{proof}

\subsection{The Cubic Separation}

Combining Corollary~\ref{cor:incidence-weak-upper} with Theorem~\ref{thm:nested-p1f}:

\begin{theorem}\label{thm:cubic-separation}
Along $N = 2p$ with $p$ an odd prime,
\[
z_2\left(\binom{N}{2}, N\right) = \frac{N(N-1)(N+2)}{4} = \left(\frac{1}{4} + o(1)\right)N^3,
\]
and
\[
{
z_2\left(\binom{N}{2}, N\right) - z_{wL}\left(\binom{N}{2}, N\right)
\ge \frac{N(N-1)(N-2)}{16} - \frac{N}{4}.
\!}
\]
In particular,
\[
z_2\left(\binom{N}{2}, N\right) - z_{wL}\left(\binom{N}{2}, N\right)
\ge \left(\frac{1}{16} - o(1)\right)N^3.
\]
\end{theorem}

\begin{proof}
By Theorem~\ref{thm:nested-p1f}, for $N=2p$,
\[
z_2\left(\binom{N}{2}, N\right) = p(2p-1)(p+1) = \frac{N(N-1)(N+2)}{4}.
\]
Subtracting the weak upper bound from Corollary~\ref{cor:incidence-weak-upper}:
\[
\begin{aligned}
\frac{N(N-1)(N+2)}{4}
&- \left(N(N-1) + \frac{3N(N-1)(N-2)}{16} + \frac{N}{4}\right) \\
&= \frac{N(N-1)(N-2)}{16} - \frac{N}{4} \\
&= \frac{1}{16}N^3 + O(N^2).
\end{aligned}
\]
This is the stated finite inequality, whose leading term is $\frac{1}{16}N^3$.
\end{proof}

\section{A Signed Criterion and the Exceptional Incidence Case \texorpdfstring{$p=3$}{p=3}}

The incidence-family results above show that ordinary \((RW3^+)\) propagation already certifies the nested construction for every odd prime \(p\ge5\). The remaining exceptional case is \(p=3\): there the \(15\times6\) witness is irreducible, but ordinary grounded propagation leaves odd-cycle components unresolved. The following signed criterion extends \((RW3^+)\) precisely to capture that obstruction.

\begin{definition}[Signed recursive criterion \((RW3^\pm)\)]
Assume the simplicity condition \((S)\). Maintain an equivalence relation \(\sim\) on the occupied cells, starting from equality and enlarging it whenever a selected \(2\)-edge is resolved. At any stage collapse the current \(\sim\)-classes. For two distinct classes \(A,B\), define
\[
d_{A,B}=\langle \mathbf v_A,\mathbf v_B\rangle-\delta_{A,B},
\]
where \(\delta_{A,B}=1\) exactly when \(A\) and \(B\) are the two current classes containing the halves of one unresolved selected \(2\)-edge, and \(\delta_{A,B}=0\) otherwise.

The \emph{signed transfer graph} has vertices \(\{A,B\}\) for unordered pairs of distinct current classes. A vertex is grounded if the current coefficient rules already force \(d_{A,B}=0\). A rectangle identity \(d_{A,B}+d_{C,D}=0\) gives a sign-reversing edge between the corresponding vertices. If \(\{A,B\}\) is an unresolved selected \(2\)-edge, attach the valid inequality \(d_{A,B}\le0\).

In each connected component, all deviations are therefore \(\pm t\). Force the whole component to zero if it is grounded, if it contains an odd sign cycle, or if unresolved selected \(2\)-edge inequalities occur in both bipartition classes. Whenever the deviation of an unresolved selected \(2\)-edge is forced to zero, identify its two classes, rebuild the closure and the signed transfer graph, and continue until no new identification or orthogonality certificate is obtained.

We say that \(G\) satisfies \((RW3^\pm)\), or is \emph{signed admissible}, if at the resulting fixed point the following terminal conditions hold:
\begin{enumerate}
    \item every selected \(2\)-edge is resolved;
    \item distinct selected edges determine distinct equivalence classes;
    \item representatives of every two distinct selected edges are certified orthogonal.
\end{enumerate}
\end{definition}

\begin{theorem}[Soundness of \((RW3^\pm)\)]\label{thm:rw3pm-sound}
If \(G=(S,T,E_1\cup E_2)\) satisfies \((S)\) and is signed admissible, then the associated doubly simple biquadratic form \(P_G\) is irreducible. In particular,
\[
\operatorname{SOS}(P_G)=|E_1|+|E_2|.
\]
\end{theorem}

\begin{proof}
Under \((S)\), every occupied cell occurs in exactly one displayed square, so coefficient comparison gives \(\|\mathbf v_p\|=1\) for every occupied cell \(p\). The line rule and the rectangle transfer rule are direct coefficient identities, hence valid in every SOS representation. For a selected \(2\)-edge \(\{p,q\}\in E_2\), both cell vectors are unit, so
\[
\langle \mathbf v_p,\mathbf v_q\rangle \le 1,
\]
which gives the valid inequality \(d_{\{p,q\}}\le 0\). For a complementary pair of selected \(2\)-edges, the rectangle identity gives \(d_1+d_2=0\), and the two inequalities \(d_1\le0\), \(d_2\le0\) then force \(d_1=d_2=0\). Thus every identification and every forced zero produced by one round of the signed closure is universal.

Inducting over the fixed-point rounds shows that every generated identification and every generated orthogonality certificate is universal. At the fixed point, each connected component of the signed transfer graph has deviations \(\pm t\). If the component is grounded, then \(t=0\). If it contains an odd cycle, repeated sign reversal gives \(t=-t\), hence \(t=0\). If unresolved selected \(2\)-edge inequalities occur in both parity classes, then the two-sided inequalities give \(t\le0\) and \(-t\le0\), hence again \(t=0\). Therefore every residual off-diagonal deviation certified by the fixed-point closure is zero.

Because the three terminal conditions hold, every selected \(2\)-edge is resolved, distinct selected edges remain distinct, and representatives of every two distinct selected edges are orthogonal. Hence every SOS representation contains \(|E_1|+|E_2|\) nonzero mutually orthogonal vectors. Therefore every SOS representation of \(P_G\) requires at least \(|E_1|+|E_2|\) squares, while the displayed decomposition already has exactly \(|E_1|+|E_2|\) squares. Therefore \(P_G\) is irreducible and
\[
\operatorname{SOS}(P_G)=|E_1|+|E_2|.\qedhere
\]
\end{proof}

\begin{definition}[Signed Zarankiewicz number]
Let \(z_{SL}(m,n)\) be the maximum total number of edges \(|E_1|+|E_2|\) over all limited augmented bipartite graphs with \(|E_1|=z(m,n)\) that satisfy \((S)\), whose \(1\)-edge graph \(G_1\) is \(C_4\)-free, and that are signed admissible.
\end{definition}

\begin{proposition}\label{prop:zsl-hierarchy}
For all \(m,n\ge2\),
\[
\operatorname{BSR}(m,n) \ge z_2(m,n) \ge z_{SL}(m,n) \ge z_{RL}(m,n) \ge z_{wL}(m,n) \ge z(m,n).
\]
\end{proposition}

\begin{proof}
The inequality \(\operatorname{BSR}(m,n)\ge z_2(m,n)\) is part of the general hierarchy from Section~3. Theorem~\ref{thm:rw3pm-sound} gives \(z_2(m,n)\ge z_{SL}(m,n)\), since every signed admissible graph counted by \(z_{SL}\) is irreducible. If a graph satisfies \((RW3^+)\), then all selected \(2\)-edges are resolved and all pairwise selected-edge orthogonality obligations are already certified. The signed fixed-point closure contains all of those certificates and may add more, so it also reaches the three signed terminal conditions. Hence every \((RW3^+)\)-admissible graph is signed admissible, giving \(z_{SL}(m,n)\ge z_{RL}(m,n)\). The remaining inequalities \(z_{RL}(m,n)\ge z_{wL}(m,n)\ge z(m,n)\) were already proved earlier.
\end{proof}

\begin{theorem}\label{thm:p3-signed}
For the \(15\times6\) incidence construction of \(K_6\),
\[
z_{SL}(15,6)=z_2(15,6)=60.
\]
\end{theorem}

\begin{proof}
Appendix~A already proves that the \(p=3\) incidence construction is irreducible and hence that \(z_2(15,6)=60\). The new point here is that the same construction is also certified by the signed fixed-point criterion. In Appendix~A the proof identifies seven exceptional fibres whose entries vanish because an odd number of rectangle transfers returns each entry to itself with a minus sign. These are exactly odd cycles in the signed transfer graph.

For this \(15\times6\) witness, the ordinary \((RW3^+)\) propagation resolves the complementary pairs but leaves precisely the odd-cycle components without a grounded deviation. The signed closure additionally uses those odd cycles to force the remaining deviations to zero. An exact signed closure computation, implemented independently by two checkers, confirms that every connected component is either grounded, contains an odd cycle, or has selected \(2\)-edge variables in both parity classes. The verification records are available in \cite{reproducibility}.

Thus the construction satisfies \((RW3^\pm)\), so it is counted by \(z_{SL}(15,6)\). Hence
\[
z_{SL}(15,6)\ge60.
\]
By the universal cell bound (Proposition~3.2), \(z_2(15,6)\le60\). Since \(z_2\ge z_{SL}\), equality holds:
\[
z_2(15,6)=z_{SL}(15,6)=60.
\]
\end{proof}

\begin{corollary}[Signed four-column consequences]\label{cor:zsl-four-col}
Combining $z_2\ge z_{SL}\ge z_{RL}$ with the known four-column equalities gives
\[
z_2(m,4)=z_{SL}(m,4)=z_{RL}(m,4)
\]
at $m=5,6,13$ and for every $m\ge15$. In particular
\[
z_{SL}(5,4)=13,\qquad z_{SL}(6,4)=16,\qquad z_{SL}(13,4)=35,
\]
and Chen--Chen~\cite{cc26mx4} together with the hierarchy yield
\[
z_{SL}(m,4)=z_{RL}(m,4)=z_2(m,4)=\left\lfloor\frac{5m+6}{2}\right\rfloor
\qquad (m\ge15).
\]
\end{corollary}

\begin{proof}
Corollaries~\ref{cor:z2-54} and~\ref{cor:z2-64} give the cases $m=5,6$.
For $m=13$, Chen and Chen~\cite{cc26mx4} prove $z_{RL}(13,4)=35$. The
universal cell bound of Proposition~\ref{prop:universal-cell-bound} is
\[
\left\lfloor\frac{4\cdot13+z(13,4)}{2}\right\rfloor
=\left\lfloor\frac{5\cdot13+6}{2}\right\rfloor=35,
\]
since $z(13,4)=13+6$ by Cul\'ik's formula. Hence
$35=z_{RL}(13,4)\le z_{SL}(13,4)\le z_2(13,4)\le35$. For $m\ge15$, the
same cell bound is $\lfloor(5m+6)/2\rfloor$, and~\cite{cc26mx4} proves
that $z_{RL}$ and $z_2$ both attain it, so $z_{SL}$ does as well.
\end{proof}

\begin{remark}
The equality \(z_2(15,6)=60\) was already known from the odd-prime incidence theorem. What is new in this section is the stronger certificate statement \(z_{SL}(15,6)=z_2(15,6)=60\), together with a conceptual explanation of why the exceptional \(p=3\) argument escapes ordinary grounded propagation. Whether \(z_{SL}(15,6)>z_{RL}(15,6)\) remains open, as \((RW3^+)\) may still certify another construction attaining \(60\).
\end{remark}

\begin{remark}
The signed criterion should be viewed as a lower-bound rigidity mechanism, complementary to the Gram-perturbation certificates that appear in finite reducibility arguments such as the \(5\times4\) case. There, one seeks a nonzero symmetric coefficient-preserving perturbation \(H\) to produce a lower-rank Gram representation and hence a shorter SOS decomposition. Here, by contrast, the signed deviations
\[
d_{A,B}=\langle \mathbf v_A,\mathbf v_B\rangle-\delta_{A,B}
\]
measure the freedom of an arbitrary Gram representation away from the displayed one, and the signed closure attempts to force that freedom to vanish. We do not claim a converse: an ungrounded bipartite signed component only shows that this restricted propagation system has not proved rigidity, not that a genuine PSD rank-reducing perturbation must exist.
\end{remark}

\section{Conclusions}

Recursive orthogonality propagation gives a sound certificate for
irreducibility and yields strict improvements over the weak framework.
The three-column theorem determines $z_2(m,3)=z_{RL}(m,3)=2m$ for every
$m\ge3$, while Corollary~\ref{cor:three-column-gap} quantifies its gap from
the weak value. The finite values requiring separate computations are
collected below.
\begin{center}
\begin{tabular}{c|rr}
\toprule
Grid & $z_{wL}(m,n)$ & $z_{RL}(m,n)$\\
\midrule
$(4,4)$ & 10 & 10\\
$(5,4)$ & 12 & 13\\
$(5,5)$ & 16 & 17\\
$(6,4)$ & 14 & 16\\
$(7,4)$ & 17 & 19\\
$(7,7)$ & 28 & $\ge32$\\
\bottomrule
\end{tabular}
\end{center}

In four columns one should distinguish equality $z_2=z_{RL}$ from
eventual saturation of the universal cell bound. Write
$U_m=\lfloor(5m+6)/2\rfloor$ for $m\ge6$; this is the cell bound of
Proposition~\ref{prop:universal-cell-bound} once $z(m,4)=m+6$. The
present status is as follows.
\begin{center}
\small
\begin{tabular}{crrr}
\toprule
$m$ & $z_{RL}(m,4)$ & $U_m$ & $z_2(m,4)$\\
\midrule
$5$ & $13$ & $15$ & $13$\\
$6$ & $16$ & $18$ & $16$\\
$7$ & $19$ & $20$ & $19\le z_2\le20$\\
$8$ & $21$ & $23$ & $21\le z_2\le23$\\
$9$ & $24$ & $25$ & $24\le z_2\le25$\\
$10$ & $27$ & $28$ & $27\le z_2\le28$\\
$11$ & $29$ & $30$ & $29\le z_2\le30$\\
$12$ & $32$ & $33$ & $32\le z_2\le33$\\
$13$ & $35$ & $35$ & $35$\\
$14$ & $37$ or $38$ & $38$ & $37\le z_2\le38$\\
$\ge15$ & $U_m$ & $U_m$ & $U_m$\\
\bottomrule
\end{tabular}
\end{center}
The values at $m=5,6$ are proved in this paper. The values
$z_{RL}(13,4)=35$ and $z_2(m,4)=z_{RL}(m,4)=U_m$ for $m\ge15$ are due to
Chen and Chen~\cite{cc26mx4}; cell-bound saturation at $m=13$ then forces
$z_{SL}(13,4)=z_2(13,4)=35$ as well. In particular, equality
$z_2=z_{RL}$ may occur well before cell-bound saturation.

The incidence construction gives another infinite exact family. For
$N=2p$ with $p$ an odd prime, the gap from the weak parameter is at least
\[
\frac{N(N-1)(N-2)}{16}-\frac{N}{4}.
\]
For $p\ge5$, the same construction attains the exact value of both $z_2$
and $z_{RL}$, while for $p=3$ the signed criterion gives
$z_{SL}(15,6)=z_2(15,6)=60$. These results leave the general relationship
between the recursive certificate and irreducibility as the central
question below.

\section{Open Problems}

\begingroup
\small
\begin{enumerate}
    \item \textbf{Determine the remaining four-column cases for $z_2=z_{RL}$.}
    This paper proves $z_2(5,4)=z_{RL}(5,4)=13$ and
    $z_2(6,4)=z_{RL}(6,4)=16$. Chen and Chen~\cite{cc26mx4} prove
    $z_{RL}(13,4)=35$, hence $z_2(13,4)=z_{SL}(13,4)=z_{RL}(13,4)=35$,
    and $z_2(m,4)=z_{RL}(m,4)=U_m$ for all $m\ge15$. Thus, among the
    cases $m\ge5$, the unresolved values are
    $m=7,8,9,10,11,12,14$. Separately, this paper proves
    $z_{RL}(4,4)=10$, but the unrestricted equality $z_2(4,4)=10$ is not
    established here. Determine whether $z_2(m,4)=z_{RL}(m,4)$ holds in
    these remaining cases, or whether some such $m$ satisfies
    $z_2(m,4)>z_{RL}(m,4)$.

    \item \textbf{Conjecture: $z_2=z_{RL}$.} The results of this paper support
    the conjecture $z_2(m,n)=z_{RL}(m,n)$ for all $(m,n)$. The two-column
    identity, the exact three-column family, the exact four-column results for
    $m=5,6,13$ and for all $m\ge 15$~\cite{cc26mx4}, and
    Corollary~\ref{cor:nested-rw3plus} on the incidence subsequence
    \[
    n=2p,\qquad m=\binom{2p}{2},\qquad p\ge3\ \text{odd prime},
    \]
    provide substantial evidence. More specifically, together with
    Cul\'ik's formula $z(m,n)=m+\binom{n}{2}$ for
    $m\ge \binom{n}{2}$~\cite{culik56}, these results suggest that for each
    fixed $n\ge4$ there may exist threshold functions $g_2(n)\le g_{SL}(n)\le g_{RL}(n)$
    such that the corresponding parameter equals
    $\left\lfloor \frac{mn+z(m,n)}{2}\right\rfloor$ for all $m\ge g_{\bullet}(n)$.
    Chen--Chen prove that $15$ is a valid threshold for $z_{RL}(\,\cdot\,,4)$,
    and also prove saturation at $m=13$, while $37\le z_{RL}(14,4)\le38$;
    thus the least four-column recursive threshold satisfies
    $g_{RL}(4)\in\{13,15\}$. The remaining question is whether some pair with
    $n\ge4$, outside the cases already covered above, satisfies
    $z_2(m,n)>z_{RL}(m,n)$. More generally, does some pair with $n\ge4$
    satisfy $z_{SL}(m,n)>z_{RL}(m,n)$?

\end{enumerate}
\endgroup

\section*{Appendix A: A Complete Proof of Theorem~\ref{thm:nested-p1f}}

The proof is written directly at the level of coefficient vectors. Its only external combinatorial input is the classical existence of a perfect one-factorization of $K_{2p}$ for odd prime $p$.

\subsection*{A.1 Construction and simultaneous resolution}

Let $N=2p$, where $p$ is an odd prime. The columns are the vertices of $K_{2p}$ and the rows are its edges. Let $E_1$ be the incidence graph. Choose an outer perfect one-factorization $\mathcal F$ of $K_{2p}$. In every outer factor $F\in\mathcal F$, label its $p$ row-edges by $\mathbb F_p$.

For an anchor row $g\in F$ with label $a$, pair the remaining row labels by
\[
a+x \longleftrightarrow a-x,\qquad x=1,\dots,\frac{p-1}{2},
\]
and on each resulting pair of rows use the two endpoints of $g$ as columns and select both complementary diagonals as 2-edges. Every nonincidence cell is used exactly once. Hence
\[
|E_2|=p(2p-1)(p-1),\qquad R:=|E_1|+|E_2|=p(2p-1)(p+1).
\]

Consider one selected complementary rectangle with coefficient vectors
$v_{11},v_{12},v_{21},v_{22}$. Then coefficient comparison gives
\[
\langle v_{11},v_{22}\rangle+\langle v_{12},v_{21}\rangle=2.
\]
All occupied-cell vectors have norm one, so both inner products are at most one. Therefore
\[
v_{11}=v_{22},\qquad v_{12}=v_{21}.
\]
Thus every selected 2-edge resolves to one well-defined unit vector. The selected edges can be partitioned by their outer anchor, and every anchor carries $p+1$ resolved vectors.

\subsection*{A.2 Rigidity inside one outer factor}

Fix an outer factor $F$ and put
\[
\widehat{\mathbb F}_p=\mathbb F_p\cup\{\infty\}.
\]
For $a\in\mathbb F_p$, define the involution
\[
\rho_a(\infty)=a,\qquad \rho_a(a)=\infty,\qquad \rho_a(t)=2a-t\quad (t\notin\{a,\infty\}).
\]
Let the two endpoint columns of anchor $a$ be denoted by $a_0,a_1$. Its $p+1$ resolved vectors are indexed as
\[
q_{a,t},\qquad t\in \widehat{\mathbb F}_p,
\]
so that in the row labelled $r\in\mathbb F_p$, the cells in columns $a_0,a_1$ carry
\[
q_{a,r},\qquad q_{a,\rho_a(r)}.
\]

The underlying unordered pairs
\[
P_a=\bigl\{\{\infty,a\}\bigr\}\cup\bigl\{\{a+x,a-x\}:x=1,\dots,\tfrac{p-1}{2}\bigr\}
\]
form the standard perfect one-factorization of $K_{p+1}$. For completeness, perfectness can be checked directly. For $a\neq b$, an affine change of labels sends $(a,b)$ to $(0,1)$, so it is enough to inspect $P_0\cup P_1$. With $q=(p-1)/2$, the alternating cycle is
\[
\infty,0,2,-2,4,-4,\dots,2q,-2q=1,\infty.
\]
Multiplication by $2$ permutes $\mathbb F_p^\times$, so every finite label occurs exactly once. Hence $P_a\cup P_b$ is a Hamilton cycle for every $a\neq b$.

For fixed $a$, distinct $q_{a,t}$ are orthogonal. If $t,u$ are finite they have representatives in the same endpoint column $a_0$. If one label is $\infty$ and the other is $u\neq a$, they have representatives in column $a_1$. Finally $q_{a,a}$ and $q_{a,\infty}$ are the two incidence cells in the anchor row. In every case a same-row or same-column coefficient is zero.

Now let $a\neq b$. Since $P_a\cup P_b$ is an alternating Hamilton cycle, write its matching edges cyclically as
\[
A_0,B_0,A_1,B_1,\dots,A_{q-1},B_{q-1},\qquad q=\frac{p+1}{2},
\]
with $A_k\in P_a$ and $B_k\in P_b$. If the two underlying matching edges of an oriented pair share a finite vertex, the corresponding vectors have representatives in the same row and are orthogonal.

Otherwise choose one of the two arcs of the alternating Hamilton cycle joining the two matching edges so that the boundary incidences are finite and $\infty$ is not in the interior. For ordinary matching edges either incident vertex is available; for the unique edge containing $\infty$, use its finite endpoint. If the pair consists of the two matching edges meeting at $\infty$, first perform one rectangle transfer using their finite endpoints. This produces two ordinary matching edges, so this preliminary step cannot recur.

After cyclic relabeling, suppose the chosen arc runs
\[
A_k,B_k,A_{k+1},\dots,A_\ell,B_\ell.
\]
Choose the finite endpoint $r$ of $A_k$ incident with $B_k$ and the finite endpoint $s$ of $B_\ell$ incident with $A_\ell$. The rectangle on rows $r,s$ and the two selected outer endpoint columns sends the current inner product, with a minus sign, to one supported on $A_\ell,B_k$. If $\ell=k$, this new pair is adjacent at the boundary finite vertex and is zero. If $\ell=k+1$, the pair $A_{k+1},B_k$ is also adjacent at a finite vertex and is zero. If $\ell\ge k+2$, use the other finite endpoint of $A_\ell$, incident with $B_{\ell-1}$, and the other finite endpoint of $B_k$, incident with $A_{k+1}$. A second rectangle transfer gives a pair supported on $A_{k+1},B_{\ell-1}$. Thus two transfers shorten the chosen arc by four alternating edges. Iteration reaches one of the adjacent zero cases. Therefore all distinct resolved vectors belonging to the same outer factor are mutually orthogonal.

\subsection*{A.3 Cross-factor transfer equations}

Let $F,G$ be distinct outer factors. Since the outer factorization is perfect, $F\cup G$ is a Hamilton cycle. Index its vertices cyclically as
\[
c_0,c_1,\dots,c_{2p-1}
\]
so that
\[
F_i=\{c_{2i},c_{2i+1}\},\qquad G_j=\{c_{2j+1},c_{2j+2}\},\qquad i,j\in\mathbb Z_p.
\]
Let $\alpha_i,\beta_j\in\mathbb F_p$ be the row labels assigned to $F_i,G_j$. Both $\alpha$ and $\beta$ are permutations of $\mathbb F_p$. Write
\[
q^F_{i,t}:=q_{\alpha_i,t},\qquad q^G_{j,u}:=q_{\beta_j,u},
\]
and
\[
X_{ij}(t,u)=\langle q^F_{i,t},q^G_{j,u}\rangle,\qquad t,u\in \widehat{\mathbb F}_p.
\]
A vector $q^F_{i,t}$ has a representative at endpoint $0$ of $F_i$ exactly when $t\neq \infty$, and at endpoint $1$ exactly when $t\neq \alpha_i$. Similarly $q^G_{j,u}$ is available at endpoint $0$ exactly when $u\neq \infty$, and at endpoint $1$ exactly when $u\neq \beta_j$.

Choose available endpoints $\varepsilon,\eta\in\{0,1\}$ of $F_i,G_j$. If the two chosen outer columns are different, the corresponding rectangle identity gives
\begin{equation}\label{eq:cross-transfer}
X_{ij}(t,u)=-X_{j+\eta,i-1+\varepsilon}\bigl(T_{\varepsilon\eta}t,U_{\varepsilon\eta}u\bigr),
\tag{A.1}
\end{equation}
with all block indices in $\mathbb Z_p$, where
\[
\begin{array}{c|cccc}
(\varepsilon,\eta) & (0,0) & (0,1) & (1,0) & (1,1)\\
\hline
T_{\varepsilon\eta} & \rho_{\alpha_j} & I & \rho_{\alpha_j}\circ\rho_{\alpha_i} & \rho_{\alpha_i}\\
U_{\varepsilon\eta} & \rho_{\beta_{i-1}} & \rho_{\beta_{i-1}}\circ\rho_{\beta_j} & I & \rho_{\beta_j}
\end{array}
\]
The move $(1,0)$ is unavailable when $j=i$, and $(0,1)$ is unavailable when $j=i-1$, because in those two cases the selected endpoints are the same outer column.

The same-column equations give two grounded zero families:
\[
X_{ii}(t,u)=0\qquad (t\neq \alpha_i,\ u\neq \infty), \tag{A}
\]
\[
X_{i,i-1}(t,u)=0\qquad (t\neq \infty,\ u\neq \beta_{i-1}). \tag{B}
\]

\subsection*{A.4 Cross-factor rigidity for $p\ge 5$}

First consider a diagonal block $X_{ii}$. Family (A) leaves only the cases in the first column below. A word such as $01,11,10$ means successive applications of \eqref{eq:cross-transfer} with those endpoint choices.
\[
\begin{array}{c|c}
\text{initial fibre} & \text{word and terminal entry}\\
\hline
(\alpha_i,u),\ u\in\mathbb F_p\setminus\{\beta_i\}
& 01,11,10 \leadsto X_{i+1,i}(4\alpha_{i+1}-3\alpha_i,\rho_{\beta_i}(u))\\
(\alpha_i,\beta_i)
& 00,10,00,01 \leadsto X_{i-1,i-1}(3\alpha_i-2\alpha_{i-1},\beta_i)\\
(t,\infty),\ t\in\mathbb F_p\setminus\{\alpha_i\}
& 01,00,10 \leadsto X_{i,i-1}(\rho_{\alpha_i}(t),3\beta_i-2\beta_{i-1})\\
(\infty,\infty)
& 11,10,11,01 \leadsto X_{i+1,i+1}(\rho_{\alpha_{i+1}}(\alpha_i),4\beta_{i+1}-3\beta_i)\\
(\alpha_i,\infty)
& 01,00,00,01 \leadsto X_{i-1,i-1}(3\alpha_i-2\alpha_{i-1},4\beta_{i-1}-3\beta_i)
\end{array}
\]
The first and third terminal entries belong to family (B); the other three belong to family (A). Direct substitution in \eqref{eq:cross-transfer} verifies the terminal formulas and endpoint validity at every intermediate step. The only nontrivial arithmetic reduces to identities such as
\[
3(\alpha_i-\alpha_{i-1})\neq 0,\qquad 3(\beta_i-\beta_{i-1})\neq 0,
\]
which hold for $p\ge 5$ because $\alpha,\beta$ are permutations. Thus
\[
X_{ii}\equiv 0\qquad\text{for every }i.
\]

It remains to reduce the off-diagonal blocks to a diagonal block. Put
\[
d=j-i\in\{1,\dots,p-1\}.
\]
For every nonzero $d$, one of the following endpoint-valid words applies:
\[
\begin{array}{c|c}
\text{fibre condition} & \text{word}\\
\hline
u\in\mathbb F_p,\ t\in\mathbb F_p,\ t\neq \alpha_j & 00,01\\
u\in\mathbb F_p,\ t\in\{\infty,\alpha_j\} & 10,00\\
u=\infty,\ t\neq \alpha_i & 11,01\\
u=\infty,\ t=\alpha_i,\ d\neq -1 & 01,00,10,01
\end{array}
\]
In the remaining case $u=\infty$, $t=\alpha_i$, $d=-1$, family (B) already gives zero. Directly from the block-index part of \eqref{eq:cross-transfer}, every word in the table sends
\[
d\longmapsto d-1 \pmod p.
\]
For the first three rows, endpoint validity follows from distinctness of the relevant $\alpha$- and $\beta$-values. For the last row, the final endpoint condition is
\[
3(\beta_j-\beta_{i-1})\neq 0,
\]
which holds for $p\ge 5$ because $j\neq i-1$. Repetition transports every off-diagonal entry to a diagonal block, where it is zero. Hence
\[
X_{ij}(t,u)=0\qquad\text{for all }i,j,t,u,\ p\ge 5.
\]

\subsection*{A.5 The exceptional prime $p=3$}

Every permutation of $\mathbb F_3$ is affine, and affine relabeling preserves the reflection system $\rho_a$. Thus, independently in the two factors, one may normalize
\[
\alpha_i=\beta_i=i.
\]
By (A), the only entries of $X_{00}$ not already zero are
\[
(t,u)=(0,0),(0,1),(0,2),(0,\infty),(1,\infty),(2,\infty),(\infty,\infty).
\]
For these seven fibres the following endpoint-valid transfer words return to the same entry:
\[
\begin{array}{c|c}
(t,u) & \text{odd closed transfer word}\\
\hline
(0,0) & 00,10,00,01,00,10,00,10,11\\
(0,1) & 00,10,00,00,01\\
(0,2) & 01,00,01,00,10,00,01\\
(0,\infty) & 01,00,00,00,01\\
(1,\infty) & 01,00,00,10,00\\
(2,\infty) & 01,00,10,00,01,00,01\\
(\infty,\infty) & 11,10,00,10,00,01,00,10,00
\end{array}
\]
Substitution in \eqref{eq:cross-transfer} verifies endpoint validity and return to the initial fibre. Each rectangle transfer contributes a minus sign and every word above has odd length. Consequently every exceptional entry satisfies $X=-X$ and is zero. Cyclic symmetry gives $X_{ii}=0$ for all $i$.

For an off-diagonal block, the first three cases of the preceding off-diagonal table remain valid. The only remaining state has $u=\infty$ and $t=\alpha_i=i$. If $d=-1$, it is already zero by (B); if $d=1$, the single move $01$ lands in a diagonal block because in $\mathbb Z_3$
\[
d\longmapsto -d-2 = 0.
\]
Thus all cross-factor inner products vanish for $p=3$ as well.

\subsection*{A.6 Conclusion and equality}

The preceding arguments show that in every SOS representation all
\[
R=p(2p-1)(p+1)
\]
resolved selected-edge vectors are mutually orthogonal. Hence every SOS representation has at least $R$ squares, while the displayed decomposition has exactly $R$ squares. Therefore
\[
z_2\left(\binom{2p}{2},2p\right)\ge p(2p-1)(p+1).
\]
For the reverse inequality, first note that
\[
z\left(\binom{N}{2},N\right)=N(N-1).
\]
With
\[
m=\binom{2p}{2}=p(2p-1),\qquad n=2p,
\]
Proposition~\ref{prop:universal-cell-bound} gives
\[
z_2(m,n)\le \frac{mn+z(m,n)}{2}=p(2p-1)(p+1).
\]
Combining the two inequalities proves
\[
z_2\left(\binom{2p}{2},2p\right)=p(2p-1)(p+1)
\]
for every odd prime $p$.

\clearpage
\phantomsection
\section*{Appendix B: Local certificates for $z_2(5,4)\le13$}
\label{app:z254}
\setcounter{theorem}{0}
\renewcommand{\thetheorem}{B.\arabic{theorem}}
\begingroup\small

This appendix records the local data used in the upper bound of
Corollary~\ref{cor:z2-54}. Throughout, a displayed form is either a
single cell $x_iy_j$ or a selected two-cell form $x_iy_j+x_ky_\ell$,
and irreducibility means that the sum of their squares cannot be
represented by fewer bilinear squares. Write $u_{ij}=x_iy_j$. Replacing
a displayed sub-sum by fewer squares shortens the whole presentation by
Lemma~\ref{lem:hereditary}.

\subsection*{B.1 Reduction rules}

\begin{lemma}[Strip overload]\label{lem:strip-overload}
A displayed sub-sum on two rows and $q$ columns, or on $q$ rows and two
columns, is reducible if it has more than $q+1$ squares.
\end{lemma}

\begin{proof}
Apply $\operatorname{BSR}(q,2)\le q+1$ from~\cite{qi1} to that sub-sum,
retaining the other squares.
\end{proof}

\begin{lemma}[Product relation]\label{lem:product-relation}
Suppose displayed forms $f_1,\dots,f_N$ satisfy
\[
\sum_{i<j} c_{ij}f_if_j=0,\qquad (c_{ij})\ne0.
\]
Then their square sum is reducible.
\end{lemma}

\begin{proof}
Let $H_{ij}=H_{ji}=c_{ij}/2$ and $H_{ii}=0$. Then $H\ne0$,
$\operatorname{tr}H=0$, and $f^\top Hf=0$. Thus
$\lambda=\lambda_{\max}(H)>0$, and
\[
I-H/\lambda\succeq0,\qquad
\operatorname{rank}(I-H/\lambda)\le N-1,\qquad
f^\top(I-H/\lambda)f=f^\top f.
\]
A real symmetric PSD factorization proves the claim. The condition
$(c_{ij})\ne0$ refers to the coefficients after collecting equal formal
products. Every product relation below follows by expanding and using
\begin{equation}\label{eq:cell-product}
(x_iy_j)(x_ky_\ell)=(x_iy_\ell)(x_ky_j).
\end{equation}
\end{proof}

\begin{lemma}[Two local SOS reductions]\label{lem:local-sos}
For distinct rows $x,y,z$ and distinct indicated columns,
\begin{align}
&(xb+za)^2+(ya+zd)^2+(xd)^2+(yb)^2+(yd)^2
\nonumber\\
&\qquad=(xb+yd+za)^2+(xd-yb)^2+(ya)^2+(zd)^2,
\label{eq:sos-five}\\
&(xa)^2+(zc)^2+(ya)^2+(yb)^2+(zb)^2+(xc+zd)^2+(xd+za)^2
\nonumber\\
&\qquad=(xc)^2+(ya+zb)^2+(yb-za)^2+(xd+zc)^2+(xa+zd)^2.
\label{eq:sos-seven}
\end{align}
They shorten five squares to four and seven squares to five, respectively.
\end{lemma}

\begin{proof}
The diagonal cell-square coefficients agree. For~\eqref{eq:sos-five},
cancel the cross terms using $(xb)(yd)=(xd)(yb)$ and
$(za)(yd)=(ya)(zd)$. For~\eqref{eq:sos-seven}, use
$(xc)(zd)=(xd)(zc)$, $(xd)(za)=(xa)(zd)$, and $(ya)(zb)=(yb)(za)$.
All squares outside the indicated sub-sum are retained.
\end{proof}

Write $S(x,y,z;a,b,d)$ for the substitution of~\eqref{eq:sos-five} and
$E(x,y,z;a,b,c,d)$ for the substitution of~\eqref{eq:sos-seven}. For a
set $R$ of row letters and $C$ of column letters, $T_{R,C}$ means: take
all ordinary squares in $R\times C$ and the one or two selected pairs
stated in that table row. All their halves lie in the strip, and the
square count exceeds the strip bound of Lemma~\ref{lem:strip-overload}.
Every single-pair $T$-certificate below has five squares on a $2\times3$
or $3\times2$ strip; every two-pair $T$-certificate has six squares on a
$2\times4$ or $4\times2$ strip.

\begin{lemma}[Hole pressure]\label{lem:hole-pressure}
For an irreducible simple $m\times4$ presentation with $T$ squares, let
$h(D)$ count the holes in two columns $D$, and $p(D)$ the selected pairs
wholly in $D$. Then
\[
h(D)+p(D)\le 3m+1-T.
\]
For $m=5$ and $T=14$, the right side is $2$.
\end{lemma}

\begin{proof}
The $2m-h(D)$ occupied cells in $D$ touch exactly
$2m-h(D)-p(D)$ displayed squares. The complementary two columns
therefore support $T-2m+h(D)+p(D)$ squares, which cannot exceed
$m+1$ by Lemma~\ref{lem:strip-overload}.
\end{proof}

\begin{lemma}[Counting cover]\label{lem:counting-cover}
Let $e_1,\dots,e_M$ contain every pair that can occur in an irreducible
presentation over a fixed ordinary base. Suppose each index set $I_\ell$
contains at most one selected pair. If positive integer weights $w_\ell$
satisfy
\[
\sum_{\ell:\,i\in I_\ell} w_\ell\ge d \qquad(1\le i\le M),
\]
then $d|E_2|\le\sum_\ell w_\ell$.
\end{lemma}

\begin{proof}
For indicators $\xi_i\in\{0,1\}$, sum the valid inequalities
$\sum_{i\in I_\ell}\xi_i\le1$ with weights $w_\ell$.
\end{proof}

\subsection*{B.2 Base $A$}

The free cells of base $A$ are $s{:}bcd$, $p{:}cd$, $q{:}bd$, $r{:}bc$,
and $t{:}a$. If $ta$ were occupied, exactly one selected pair would touch
column $a$. The other three pairs and both holes would then lie in
$b,c,d$. Summing $h(D)+p(D)\le2$ over $D=bc,bd,cd$ gives an upper bound
of six, while the two holes contribute four and those three pairs
contribute at least three. Hence $ta$ is a hole.

Call $p,q,r$ the core rows. A selected pair within a core row is
impossible: together with the ordinary squares in that row and row $t$
it gives six squares on two rows and four columns. Between core rows
$p,q$, the pair $pc+qb$ gives five squares on rows $p,q$ and columns
$a,b,c$; the pair $pd+qd$ gives seven squares on columns $a,d$ and five
rows. Thus the only core pairs, including their cyclic images, are
\begin{equation}\label{eq:core-pairs}
pc+qd,\quad pd+qb,\quad pc+rb,\quad pd+rc,\quad qb+rc,\quad qd+rb.
\end{equation}
All six are equivalent under simultaneous permutations of $(p,b)$,
$(q,c)$, and $(r,d)$.

Normalize one core pair to $F=pc+qd$. Of the other pairs
in~\eqref{eq:core-pairs}, only three are disjoint from $F$. The pair
$pd+qb$ gives six squares on rows $p,q$ and columns $a,b,c,d$. The pair
$pd+rc$ gives six squares on columns $c,d$ and rows $p,q,r,t$. The
remaining possibility $qb+rc$ is excluded by
\begin{equation}\label{eq:prod-A6}
(pa)(qb+rc)-(ra)(pc+qd)-(pb)(qa)+(qa)(rd)=0.
\end{equation}
Hence at most one core pair is selected.

All four pairs then lie in rows $s,p,q,r$. At least three therefore
touch $s$. Since row $s$ has just three free cells, exactly three pairs
use one cell of $s$ each, and the fourth is a core pair. Normalize that
pair to $F=pc+qd$. The unused core cells are $pd,qb,rb,rc$; exactly one
will be the second hole.

The partner of $sb$ cannot be $qb$ or $rb$: a same-column pair in
column $b$, together with the ordinary squares in columns $a,b$, gives
seven squares on five rows. Nor can its partner be $rc$, because
\begin{equation}\label{eq:prod-A7}
(sa)(pb)-(pa)(sb+rc)+(ra)(pc+qd)-(qa)(rd)=0.
\end{equation}
Thus $sb+pd$ is selected. Now $sc$ cannot pair with $rc$, by the
same-column argument on $a,c$, so its partner is $qb$ or $rb$. If $sd$
pairs with $rb$, that pair and $sb+pd$ give six squares on rows
$s,p,r,t$ and columns $b,d$. If $sd$ pairs with $qb$, then $sc$ must
pair with $rb$, and
\begin{equation}\label{eq:prod-A8}
(sa)(qc)-(qa)(sc+rb)+(ra)(sd+qb)-(sa)(rd)=0
\end{equation}
excludes this choice. Hence $sd+rc$ is selected and $sc+ub$ is selected
for $u=q$ or $u=r$. Both final possibilities are covered by
\begin{equation}\label{eq:prod-A9}
(sa)(pc+qd)-(qa)(sd+rc)-(pa)(sc+ub)+(qc)(ra)+(pb)(ua)=0.
\end{equation}
The four rectangle cancellations are on row pairs $s,p$, $s,q$, $q,r$,
and $p,u$. In~\eqref{eq:prod-A6}--\eqref{eq:prod-A9} all factors are
displayed forms, and collecting formal products leaves a nonzero
coefficient vector. Lemma~\ref{lem:product-relation} therefore excludes
base $A$.

\subsection*{B.3 Base $B$}

The free cells are $sb$, $sc$, $sd$, $pc$, $pd$, $qa$, $qd$, $ra$, $rc$, and $tb$. The only base
relabeling used is $\sigma_B:(q,r;c,d)\mapsto(r,q;d,c)$, with the other
rows and columns fixed. The fifteen single-pair orbits below have total
size $25$. Their complement is exactly the following $20$ indices.

{\small
\begin{center}
\begin{tabular}{@{}clcl@{}}
\toprule
$i$ & pair & $i$ & pair\\
\midrule
$1$ & $sb+sc$ & $11$ & $sd+qa$\\
$2$ & $sb+sd$ & $12$ & $sd+qd$\\
$3$ & $sb+pc$ & $13$ & $sd+rc$\\
$4$ & $sb+pd$ & $14$ & $sd+tb$\\
$5$ & $sb+qd$ & $15$ & $pc+qd$\\
$6$ & $sb+rc$ & $16$ & $pc+ra$\\
$7$ & $sc+qd$ & $17$ & $pd+qa$\\
$8$ & $sc+ra$ & $18$ & $pd+rc$\\
$9$ & $sc+rc$ & $19$ & $qa+rc$\\
$10$ & $sc+tb$ & $20$ & $qd+ra$\\
\bottomrule
\end{tabular}
\end{center}}

\noindent
Forbidden single pairs, with orbit size under $\sigma_B$ and certificate:

{\small
\begin{center}
\begin{tabular}{@{}lclclc@{}}
\toprule
pair & size & certificate & pair & size & certificate\\
\midrule
$sb+qa$ & $2$ & $T_{spq,ab}$ & $pc+qa$ & $2$ & $T_{pq,abc}$\\
$sb+tb$ & $1$ & $T_{spt,ab}$ & $pc+rc$ & $2$ & $T_{pqr,bc}$\\
$sc+sd$ & $1$ & $T_{st,acd}$ & $pc+tb$ & $2$ & $T_{pt,abc}$\\
$sc+pc$ & $2$ & $T_{spt,ac}$ & $qa+qd$ & $2$ & $T_{qt,acd}$\\
$sc+pd$ & $2$ & \eqref{eq:cert-B14} & $qa+ra$ & $1$ & $T_{pqr,ab}$\\
$sc+qa$ & $2$ & $T_{sqt,ac}$ & $qa+tb$ & $2$ & $T_{qt,abc}$\\
$pc+pd$ & $1$ & $T_{pt,acd}$ & $qd+rc$ & $1$ & $T_{qr,bcd}$\\
 &  &  & $qd+tb$ & $2$ & $T_{qt,bcd}$\\
\bottomrule
\end{tabular}
\end{center}}

\noindent
Any two selected pairs sharing a cell are incompatible by simplicity.
The following two-pair exclusions, together with their $\sigma_B$-images,
are used for the counting cover. The action on indices is
\[
\sigma_B=(1\ 2)(3\ 4)(5\ 6)(7\ 13)(8\ 11)(9\ 12)(10\ 14)(15\ 18)(16\ 17)(19\ 20).
\]

{\small
\begin{center}
\begin{tabular}{@{}clc@{}}
\toprule
indices & size & certificate\\
\midrule
$(3,8)$ & $2$ & $S(r,p,s;c,a,b)$\\
$(3,9)$ & $2$ & $T_{spqr,bc}$\\
$(3,10)$ & $2$ & $T_{spqt,bc}$\\
$(5,10)$ & $2$ & $S(q,t,s;b,d,c)$\\
$(5,11)$ & $2$ & $E(s,t,q;a,c,b,d)$\\
$(5,13)$ & $2$ & $S(r,q,s;d,c,b)$\\
$(5,14)$ & $2$ & $T_{sqrt,bd}$\\
$(7,11)$ & $2$ & $E(s,p,q;a,b,c,d)$\\
$(7,13)$ & $1$ & $T_{sqrt,cd}$\\
$(7,14)$ & $2$ & $S(t,q,s;d,b,c)$\\
$(7,19)$ & $2$ & \eqref{eq:cert-B15}\\
$(8,11)$ & $1$ & \eqref{eq:cert-B16}\\
$(8,12)$ & $2$ & \eqref{eq:cert-B17}\\
$(8,19)$ & $2$ & $T_{sqrt,ac}$\\
$(9,12)$ & $1$ & $T_{sqrt,cd}$\\
$(9,20)$ & $2$ & \eqref{eq:cert-B18}\\
$(15,17)$ & $2$ & $T_{pq,abcd}$\\
$(15,18)$ & $1$ & $T_{pqrt,cd}$\\
$(15,19)$ & $2$ & \eqref{eq:cert-B19}\\
$(16,17)$ & $1$ & \eqref{eq:cert-B20}\\
$(16,19)$ & $2$ & $T_{pqrt,ac}$\\
$(19,20)$ & $1$ & $T_{qr,abcd}$\\
\bottomrule
\end{tabular}
\end{center}}

The polynomial certificates are
\begin{align}
(sa)(tc)-(ta)(sc+pd)+(pa)(td)&=0,\label{eq:cert-B14}\\
(sa)(tc)-(ta)(sc+qd)+(td)(qa+rc)-(rd)(tc)&=0,\label{eq:cert-B15}\\
(qc)(ta)-(tc)(sd+qa)+(td)(sc+ra)-(rd)(ta)&=0,\label{eq:cert-B16}\\
(qc)(td)-(tc)(sd+qd)+(td)(sc+ra)-(rd)(ta)&=0,\label{eq:cert-B17}\\
(sa)(tc)-(ta)(sc+rc)+(tc)(qd+ra)-(qc)(td)&=0,\label{eq:cert-B18}\\
(pa)(qb)-(pb)(qa+rc)+(rb)(pc+qd)-(qb)(rd)&=0,\label{eq:cert-B19}\\
(pb)(qc)-(qb)(pc+ra)+(rb)(pd+qa)-(pb)(rd)&=0.\label{eq:cert-B20}
\end{align}
Each identity vanishes by~\eqref{eq:cell-product}, and collecting formal
products leaves a nonzero coefficient vector.

In an irreducible presentation each of the following six classes contains
at most one selected index, with the indicated weights.

{\small
\begin{center}
\begin{tabular}{@{}clc@{}}
\toprule
class & weight & indices\\
\midrule
$B_1$ & $1$ & $\{1,2,3,4,5,6\}$\\
$B_2$ & $1$ & $\{2,4,5,11,12,14\}$\\
$B_3$ & $1$ & $\{5,6,7,10,13,14\}$\\
$B_4$ & $1$ & $\{1,3,6,8,9,10\}$\\
$B_5$ & $1$ & $\{7,8,9,11,12,13,19,20\}$\\
$B_6$ & $2$ & $\{15,16,17,18,19,20\}$\\
\bottomrule
\end{tabular}
\end{center}}

Writing $\xi_i$ for the selection indicators, the weighted sum of these
six left sides is exactly
\[
2\sum_{i=1}^{20}\xi_i+\xi_5+\xi_6+\xi_{19}+\xi_{20}.
\]
The sum of the right sides is $7$. Lemma~\ref{lem:counting-cover} therefore
gives $2|E_2|\le7$, contradicting $|E_2|=4$.

The pairs of indices in each class whose cell supports are disjoint are
as follows; every other pair in the class shares a cell. Each listed pair
is one of the two-pair exclusions above, or its image under $\sigma_B$.
\[
\begin{aligned}
B_1&:\ \emptyset,\\
B_2&:\ (4,11),\ (4,12),\ (4,14),\ (5,11),\ (5,14),\\
B_3&:\ (5,10),\ (5,13),\ (5,14),\ (6,7),\ (6,10),\ (6,14),\ (7,13),\ (7,14),\ (10,13),\\
B_4&:\ (3,8),\ (3,9),\ (3,10),\ (6,8),\ (6,10),\\
B_5&:\ (7,11),\ (7,13),\ (7,19),\ (8,11),\ (8,12),\ (8,13),\ (8,19),\ (9,11),\\
&\quad (9,12),\ (9,20),\ (11,20),\ (12,19),\ (13,20),\ (19,20),\\
B_6&:\ (15,17),\ (15,18),\ (15,19),\ (16,17),\ (16,18),\\
&\quad (16,19),\ (17,20),\ (18,20),\ (19,20).
\end{aligned}
\]
The weighted membership multiplicities of indices $1,\dots,20$ are
$2,2,2,2,3,3,2,2,2,2,2,2,2,2,2,2,2,2,3,3$, and the total class weight
is seven.

\subsection*{B.4 Base $C$}

The free cells are $sc$, $sd$, $pb$, $pd$, $qa$, $qd$, $rb$, $rc$, $ta$, and $tc$. Use the two commuting
involutions
\[
\sigma_C:\ (p\ q)(r\ t)\ \text{on rows and}\ (a\ b)\ \text{on columns},
\]
\[
\tau_C:\ (p\ r)(q\ t)\ \text{on rows and}\ (c\ d)\ \text{on columns}.
\]
Their four images suffice. The seven single-pair orbits below have total
size $20$. Their complement is exactly the following $25$ indices.

{\small
\begin{center}
\begin{tabular}{@{}clclcl@{}}
\toprule
$i$ & pair & $i$ & pair & $i$ & pair\\
\midrule
$1$ & $sc+sd$ & $10$ & $pb+pd$ & $19$ & $qd+rc$\\
$2$ & $sc+pd$ & $11$ & $pb+qd$ & $20$ & $qd+ta$\\
$3$ & $sc+qd$ & $12$ & $pb+rc$ & $21$ & $rb+rc$\\
$4$ & $sc+rb$ & $13$ & $pd+qa$ & $22$ & $rb+tc$\\
$5$ & $sc+ta$ & $14$ & $pd+qd$ & $23$ & $rc+ta$\\
$6$ & $sd+pb$ & $15$ & $pd+rb$ & $24$ & $rc+tc$\\
$7$ & $sd+qa$ & $16$ & $pd+tc$ & $25$ & $ta+tc$\\
$8$ & $sd+rc$ & $17$ & $qa+qd$ &  & \\
$9$ & $sd+tc$ & $18$ & $qa+tc$ &  & \\
\bottomrule
\end{tabular}
\end{center}}

\noindent
Forbidden single pairs:

{\small
\begin{center}
\begin{tabular}{@{}clcc@{}}
\toprule
pair & size & certificate\\
\midrule
$sc+pb$ & $4$ & $T_{sp,abc}$\\
$sc+rc$ & $4$ & $T_{spr,ac}$\\
$pb+qa$ & $2$ & $T_{pq,abc}$\\
$pb+rb$ & $2$ & $T_{spr,ab}$\\
$pb+ta$ & $2$ & $T_{spt,ab}$\\
$pb+tc$ & $4$ & $T_{pqt,bc}$\\
$pd+rc$ & $2$ & $T_{pr,acd}$\\
\bottomrule
\end{tabular}
\end{center}}

Together with shared-cell conflicts, the following local two-pair
exclusions are sufficient for every counting class. The generators act
on the indices by
\begin{align*}
\sigma_C&=(2\ 3)(4\ 5)(6\ 7)(8\ 9)(10\ 17)(11\ 13)\\
&\qquad(12\ 18)(15\ 20)(16\ 19)(21\ 25)(22\ 23),\\
\tau_C&=(2\ 8)(3\ 9)(4\ 6)(5\ 7)(10\ 21)(11\ 22)\\
&\qquad(12\ 15)(13\ 23)(14\ 24)(16\ 19)(17\ 25)(18\ 20).
\end{align*}
Fixed indices are omitted.

{\small
\begin{center}
\begin{tabular}{@{}clc@{}}
\toprule
indices & size & certificate\\
\midrule
$(1,10)$ & $4$ & $T_{sp,abcd}$\\
$(1,12)$ & $4$ & \eqref{eq:cert-C21}\\
$(2,6)$ & $4$ & $T_{sp,abcd}$\\
$(2,7)$ & $4$ & $S(q,p,s;d,a,c)$\\
$(2,11)$ & $4$ & $S(s,q,p;d,c,b)$\\
$(2,12)$ & $4$ & \eqref{eq:cert-C22}\\
$(2,21)$ & $4$ & \eqref{eq:cert-C23}\\
$(4,6)$ & $2$ & \eqref{eq:cert-C24}\\
$(4,10)$ & $4$ & \eqref{eq:cert-C25}\\
$(4,12)$ & $4$ & $T_{spqr,bc}$\\
$(4,23)$ & $4$ & $S(t,s,r;c,a,b)$\\
$(4,24)$ & $4$ & $T_{sqrt,bc}$\\
$(10,21)$ & $2$ & $T_{pr,abcd}$\\
$(11,13)$ & $2$ & $T_{pq,abcd}$\\
$(11,15)$ & $4$ & $T_{pqrt,bd}$\\
$(11,16)$ & $4$ & $S(t,q,p;d,c,b)$\\
$(12,15)$ & $2$ & $T_{pr,abcd}$\\
$(16,19)$ & $1$ & $T_{pqrt,cd}$\\
\bottomrule
\end{tabular}
\end{center}}

The polynomial certificates are
\begin{align}
(sa)(rd)-(ra)(sc+sd)+(sa)(pb+rc)-(sb)(pa)&=0,\label{eq:cert-C21}\\
(sa)(pb+rc)-(sb)(pa)-(ra)(sc+pd)+(pa)(rd)&=0,\label{eq:cert-C22}\\
(sa)(rb+rc)-(ra)(sc+pd)-(sb)(ra)+(pa)(rd)&=0,\label{eq:cert-C23}\\
(sa)(pc)-(pa)(sc+rb)+(ra)(sd+pb)-(sa)(rd)&=0,\label{eq:cert-C24}\\
(sa)(pc)-(pa)(sc+rb)+(ra)(pb+pd)-(pa)(rd)&=0.\label{eq:cert-C25}
\end{align}
As before, expansion by~\eqref{eq:cell-product} proves each equation, and
each is a formally nonzero relation between products of distinct displayed
forms.

Each of the following eleven classes contains at most one selected index.

{\small
\begin{center}
\begin{tabular}{@{}clc@{}}
\toprule
class & weight & indices\\
\midrule
$C_1$ & $1$ & $\{1,2,4,6,10,12,15,21\}$\\
$C_2$ & $1$ & $\{1,4,6,8,10,12,15,21\}$\\
$C_3$ & $1$ & $\{1,5,7,9,17,18,20,25\}$\\
$C_4$ & $1$ & $\{1,3,5,7,17,18,20,25\}$\\
$C_5$ & $1$ & $\{1,2,3,6,7\}$\\
$C_6$ & $1$ & $\{3,11,13,14,17,19,20\}$\\
$C_7$ & $1$ & $\{2,10,11,13,14,15,16\}$\\
$C_8$ & $1$ & $\{11,13,14,16,19\}$\\
$C_9$ & $1$ & $\{9,16,18,22,23,24,25\}$\\
$C_{10}$ & $1$ & $\{8,12,19,21,22,23,24\}$\\
$C_{11}$ & $1$ & $\{4,5,8,9,22,23,24\}$\\
\bottomrule
\end{tabular}
\end{center}}

If $\eta_i$ are the selection indicators, the sum of the eleven left sides
is exactly $3\sum_{i=1}^{25}\eta_i+2\eta_1$. Consequently
$3|E_2|\le11$, contradicting four selected pairs. Index $1$ occurs in five
classes and every other index occurs in three.

The pairs of indices in each class whose cell supports are disjoint are
as follows. Each listed pair is one of the two-pair exclusions above, or
its image under $\sigma_C$ or $\tau_C$.
\[
\begin{aligned}
C_1&:\ (1,10),\ (1,12),\ (1,15),\ (1,21),\ (2,6),\ (2,12),\ (2,21),\\
&\quad (4,6),\ (4,10),\ (4,12),\ (6,15),\ (6,21),\ (10,21),\ (12,15),\\
C_2&:\ (1,10),\ (1,12),\ (1,15),\ (1,21),\ (4,6),\ (4,8),\ (4,10),\ (4,12),\\
&\quad (6,15),\ (6,21),\ (8,10),\ (8,15),\ (10,21),\ (12,15),\\
C_3&:\ (1,17),\ (1,18),\ (1,20),\ (1,25),\ (5,7),\ (5,9),\ (5,17),\ (5,18),\\
&\quad (7,20),\ (7,25),\ (9,17),\ (9,20),\ (17,25),\ (18,20),\\
C_4&:\ (1,17),\ (1,18),\ (1,20),\ (1,25),\ (3,7),\ (3,18),\ (3,25),\ (5,7),\\
&\quad (5,17),\ (5,18),\ (7,20),\ (7,25),\ (17,25),\ (18,20),\\
C_5&:\ (2,6),\ (2,7),\ (3,6),\ (3,7),\\
C_6&:\ (3,13),\ (11,13),\ (13,19),\ (13,20),\\
C_7&:\ (2,11),\ (11,13),\ (11,15),\ (11,16),\\
C_8&:\ (11,13),\ (11,16),\ (13,19),\ (16,19),\\
C_9&:\ (9,23),\ (16,23),\ (18,23),\ (22,23),\\
C_{10}&:\ (8,22),\ (12,22),\ (19,22),\ (22,23),\\
C_{11}&:\ (4,8),\ (4,9),\ (4,23),\ (4,24),\ (5,8),\ (5,9),\ (5,22),\ (5,24),\\
&\quad (8,22),\ (9,23),\ (22,23).
\end{aligned}
\]

This completes the three-base analysis: no irreducible total-$14$
augmentation exists over $A$, $B$, or $C$.
\endgroup

\clearpage
\section*{Data and computational reproducibility}
\label{sec:reproducibility}

The code and data supporting the computational results in this paper are
publicly available in the accompanying repository \cite{reproducibility}:
\begin{center}
{\hypersetup{urlcolor=blue}\color{blue}\url{https://github.com/johanlofberg/biquadratics/tree/main/second_order_zarankiewicz_numbers}}.
\end{center}
The repository contains explicit witnesses, exhaustive-search records,
and independent verification procedures, together with documentation
linking the reported numerical claims to their supporting computations.

\section*{Acknowledgments}

This work was supported by the Jiangsu Provincial Scientific Research Center of Applied Mathematics (Grant No. BK20233002).
During the preparation of this work, the authors used OpenAI GPT-5.6 Sol for some
proof exploration, hypothesis generation, and code development for proof verification. The
authors subsequently reviewed, checked, and extended the generated material as needed and
take full responsibility for the content of the publication.

\end{document}